\documentclass{article}

\usepackage[margin=3cm]{geometry}
\usepackage{amsmath,amsfonts,graphicx,subfigure,amsthm}
\usepackage[ruled,linesnumbered]{algorithm2e}
\usepackage{cite}

\newtheorem{theorem}{Theorem}[section]

\newtheorem{lemma}[theorem]{Lemma}
\newtheorem{proposition}{Proposition}

\newtheorem{remark}{Remark}
\providecommand{\keywords}[1]{\textbf{\textit{Keywords---}} #1}

\title{\bf Path Planning in Unknown Environments\\ Using Optimal Transport Theory}
\author{Haoyan Zhai  \and Magnus Egerstedt \and Haomin Zhou}
\date{}

\begin{document}
\maketitle

\begin{abstract}
This paper introduces 
a graph-based, potential-guided method for path planning problems in unknown environments, where obstacles are unknown until the robots are in close proximity to the obstacle locations.
Inspired by optimal transport theory, the proposed method generates a graph connecting the initial and target configurations, and then finds a path over the graph using the available environmental information.
The graph and path are updated iteratively when newly encountered obstacle information becomes available. The resulting method is a deterministic procedure proven to be complete, i.e., 
it is guaranteed to find a feasible path, when one exists, in a finite number of iterations. The method is scalable to high-dimensional problems. In addition, our method does not search the entire domain for the path, instead, the algorithm only explores a sub-region that can be described by the evolution of the Fokker-Planck equation. We demonstrate the performance of our algorithm via several numerical examples with different environments and dimensions, including high-dimensional cases.

\end{abstract}

\section{Introduction}
This paper considers path planning for a robot, or possibly a group of robots, in an unknown environment. In other words, a set of robots in given initial configurations are tasked with finding a feasible path to the target configuration, while avoiding collisions with obstacles.
We consider scenarios where the number of robots is fixed and where obstacles are detected when they are within detection range to one of the robots in the group. We assume that the system employs a broadcast strategy in the sense that
the obstacle information, once available, is shared among the group immediately.
Compared to the path planning problem for known environments, there are several significant challenges when the problem is posed in unknown environments. First of all, re-planning while moving becomes necessary when a planned path is blocked by newly detected obstacles. This raises livelock concerns, i.e.,  the robots may end up moving in loops, and never reach the target even when there exist feasible paths. 
Secondly, the configuration space may be quite high-dimensional. Especially when there are a large number of robots present. As a result,  
grid-like discretizations often lead to intractable computations. In this case, working with graphs is a viable option to reduce the computation burden. However, the cost can still be high if the graph has to span everywhere in the high-dimensional space. Thirdly, there may be narrow pathways between obstacles, which poses significant hurdles to identify them in the search process.  Finally, for problems with unknown environments, optimality is only meaningful in the currently known environments. Hence, one may have to accept locally optimal, or even simply feasible, solutions in some cases. 

There exists an extensive literature on path planning.
For example, the well-known Probabilistic Road Map (PRM) method generates a random graph that does not intersect with obstacles and then finds a path over the graph to connect the initial and target configurations \cite{overmars1992random, kavraki1994randomized, amato1996randomized, svestka1997robot}. PRM guarantees a connection between the initial and target configurations when the graph is dense enough in the configuration space. Many additional PRM have been reported in the past decades, see \cite{hsu1997path,barraquand1997random,branicky2001quasi,boor1999gaussian} for details.

The Rapidly-exploring Random Tree (RRT) is another influential method  \cite{lavalle1998rapidly}. It creates a tree structure rooted at the initial configuration, making sure that all the vertices are connected to the initial one. At each step, the algorithm picks a configuration in the space randomly, and checks whether it can be added to the tree following certain criteria. This is continued until the target or a configuration close enough to the target is included in the tree. 
This algorithm has been adopted to path finding in unknown environments \cite{tian2007application,yang2013gaussian} and has recent improvements such as RRT$^*$ are reported in  \cite{karaman2011sampling,noreen2016optimal}.

The Artificial Potential Field (APF) assigns the robot a positive charge while the target configuration a negative one \cite{khatib1986real} . This drives the robot towards the target. To avoid collisions with the obstacles, APF sets the obstacles with positive charges that repel the robot.  Since it is a local gradient method, APF is amenable to high dimensional problems. However, a potential limitation of the original APF is the creation of unnecessary local minima due to the presence of obstacles, which may fail the algorithm. In recent years, there are reported improvements of APF, see \cite{sfeir2011improved,konrad2018secant,chen2016uav,montiel2015path} and the references therein.

The family of Bug Algorithms, starting from the original Bug0, Bug1 and Bug2 \cite{lumelsky1986dynamic,lumelsky1987path} to the later developments, such as TangentBug \cite{kamon1996new}, DistBug \cite{kamon1997sensory} and many other variants, adopt two basic modes as their design principle: motion-to-goal mode and boundary-following mode. They are powerful tools, with theoretical convergence guarantees, especially suitable for path planning in unknown environments in $2$ dimensional working space. Some recent survey and performance comparison studies can be found in \cite{mcguire2018comparative,ng2007performance}. 



In addition, widely known graph based methods, like $A^*$ \cite{doran1966experiments}, $D^*$ \cite{stentz1994optimal}, Focused $D^*$ \cite{stentz1995focussed} and $D^*$ lite \cite{koenig2005fast} can be used for path planning in both known and unknown environments \cite{podsedkowski2001new}. When applied to the problems in unknown environments, they often require to cover the entire region by discrete lattice grids, on which the algorithms are performed to find paths. In literature, there are other types of algorithms such as genetic algorithm \cite{walker2002comparison, lei2006improved}, evolutionary programming \cite{contreras2015mobile}, fuzzy logic \cite{hassanzadeh2009path,wang2005fuzzy}, neural network \cite{luo2008bioinspired}, network simplex method \cite{ersson2001path}, method of evolving junctions \cite{li2017method}, fast marching tree \cite{janson2015fast}, and a few hybrid approaches that combine different methods \cite{dolgov2010path, van2006anytime}, and also many more swarm strategies for multi-agent systems in recent years \cite{wagner2015subdimensional,cui2016mutual,das2016hybridization}.

In this paper, we present a potential guided, graph-based pathfinding method inspired by the evolution of Fokker-Planck Equation (FPE) in optimal transport theory \cite{villani2008optimal}. Optimal transport theory is a branch of mathematics studying how to transport one probability distribution to another with the optimal cost. There are different ways to formulate the theory such as using linear programming \cite{kantarovich1939mathematical}, or partial differential equations (PDEs) \cite{brenier1987decomposition, brenier1991polar} etc. Among various formulations, the optimal control approach reveals that the FPE is the gradient flow of a free energy, consisting of potential energy and entropy, in the probability space equipped with the so-called Wasserstein distance \cite{otto2001geometry}. Using the optimal transport theory, one can show that FPE can escape the traps of any local minima in a potential field and reach the Gibbs distribution which concentrates on the global minimizer. Incorporating this property and advantages of several existing algorithms, we design a novel method for path planning in unknown environments. Our goal is providing an alternative algorithm that can work efficiently, especially for problems with high dimensional configuration spaces, such as multi-agent systems. When designing our algorithm, we introduce a potential field, defined by the distance to the target configurations in this paper. The unique global minimum of the potential field is at the target configurations. Unlike APF, the obstacles do not contribute to the potential field, instead, they define the infeasible regions. We generate a graph that has a tree structure originated from the initial configuration, growing in a deterministic manner guided by the flow direction of FPE towards the target configuration. Our algorithm has the following features: 
\begin{enumerate}
\item The algorithm is a graph based deterministic procedure with guaranteed convergence properties, meaning that the algorithm terminates in a finite number of steps, and returns a feasible path if there exists one. Thus, the algorithm is complete. We would highlight that the convergence of our algorithm is deterministic, in contrast to the asymptotic convergence results shared by many methods using randomness.

\item If the algorithm does not find a feasible path, one concludes that from initial to target configurations, there does not exist a feasible trajectory such that one can find a tubular region, centered at the trajectory with a small radius, not intersecting with obstacles. The lower bound of the radius for the tubular region is proportional to the step size used in graph generation.

\item The path found by the algorithm is locally optimal in the known environment up to the current location of the robots.

\item The graph generated by the algorithm has a tree structure growing linearly with respect to the dimension of the configuration space. Together with the dimension reduction techniques proposed to rapidly escape the local traps, the algorithm can efficiently handle high dimensional problems.

\item The algorithm only explores a limited region defined by the solution of a FPE, even when the obstacles are not known {\it a priori}.
\end{enumerate}

It should be noted that optimal transport theory has been considered in several recent studies for path planning. For example, swarming robots are modeled by a distribution, and their optimal transport map is calculated by linear programming in \cite{bandyopadhyay2014probabilistic}.
Another paper presents a partial differential equation (PDE) based swarming model to the deployment of a large scale of robots using Kantorovick-Rubinstein relaxation of optimal transport theory  \cite{krishnan2018distributed}. Our method is different in that we directly use the evolution of FPE to guide the path construction. 

In the next section, we present the details of the algorithm with the finite step stopping property. In Section \ref{example}, we show some numerical examples to illustrate the performance in both low and high dimensional configuration spaces. Section \ref{speed_up} gives strategies for dimension reduction near local minima to further lower the computational cost. In Section \ref{FPE} the relationship between the algorithm and optimal transport theory is discussed. The convergence proof is given in Section \ref{convergence}. We end the paper with a brief conclusion in the last section. 

\section{Algorithm} \label{algorithm}
Let the configuration space $\Omega$ be a bounded connected domain in $\mathbb{R}^n$. We assume that the robots can alter configurations freely in $\Omega$ as long as the change does not violate the required constraints. There are two types of constraints we consider in this paper. One is the constraints known in advance, for example, two robots can't be too close or too far away from each other in the multi-agent system. We denote those constraints by
\[
\phi=(\phi_1,\phi_2\cdots,\phi_{k_1}):\Omega\longrightarrow\mathbb{R}^{k_1},
\]
and a configuration $x\in\Omega$ does not satisfy the constraints when $\phi_i(x)<0$ for some $i\in\{1,\cdots,k_1\}$. The other type of constraints is given by unknown environments, such as unknown obstacles. We represent them by
\[
\psi=(\psi_1,\psi_2\cdots,\psi_{k_1}):\Omega\longrightarrow\mathbb{R}^{k_2},
\]
and $\psi_i(x)<0$ for some $i\in\{1,\cdots,k_2\}$ means the constraints being violated. We emphasise that $\psi(x)$ can only be detected if robots are close enough to the obstacles. This implies that the knowledge of $\psi(x)$ must be updated dynamically while the robots are in motion. For the simplicity in discussion, we assume that both $\phi$ and $\psi$ are continuous.

To illustrate the setups, we give a single robot example in Figure \ref{fig:intro}. The configuration space is a square, all the gray bars are the obstacles that the robot cannot collide with. The light gray bars in the figure are obstacles undetected. Like in the second picture in Figure \ref{fig:intro}, if the robot moves horizontally but not too far away from its initial configuration, there is no detected obstacle. As the robot moves, more and more obstacles are recognized. Our goal is finding a path from the initial configuration $x_s$ (red diamond in Figure \ref{fig:intro}) to the target configuration $x_t$ (red circle in Figure \ref{fig:intro}). More precisely, we want to find a feasible path 
\[
\gamma(t): [0,T]\rightarrow\Omega,
\]
satisfying $\phi(\gamma(t))\geq0$ and $\psi(\gamma(t))\geq0$ for all $t\in[0,T]$, such as the red path in Figure \ref{fig:intro}, while $\psi(x)$ is updated with newly detected obstacles as $\gamma(t)$ changes. 

To describe the dynamical change of the unknown constraints while moving along a path, we mark a configuration $x$ as part of the detected obstacles if ($\psi_i(x)<0$) and $x$ is within distance $R$ to the current configuration of the robot. To be precise, we define 
\begin{align}
\widetilde{\psi}(x,t,\gamma)=\left\lbrace
\begin{array}{ll}
\psi(x)&\text{if }\|x-\gamma(\tau)\|\leq R\text{ for some }\tau\leq t\\
0&\text{otherwise}
\end{array}\right.
\label{equ:envupdate}
\end{align}
as the detected part of the environment along the path.

\begin{figure}\centering
\subfigure[] { \label{fig:intro1}
\includegraphics[width=0.31\columnwidth]{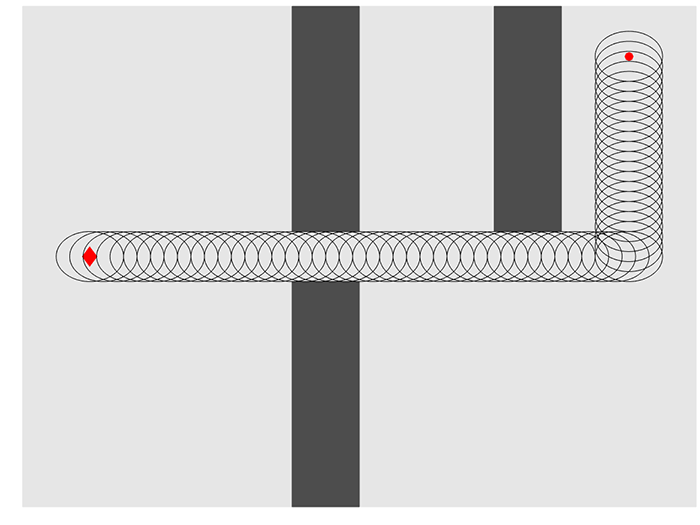} 
}
\subfigure[] { \label{fig:intro2}
\includegraphics[width=0.31\columnwidth]{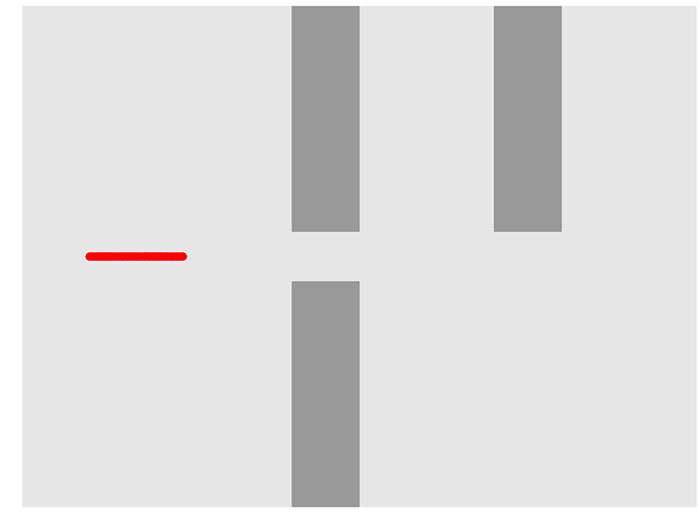} 
}
\subfigure[] { \label{fig:intro3}
\includegraphics[width=0.31\columnwidth]{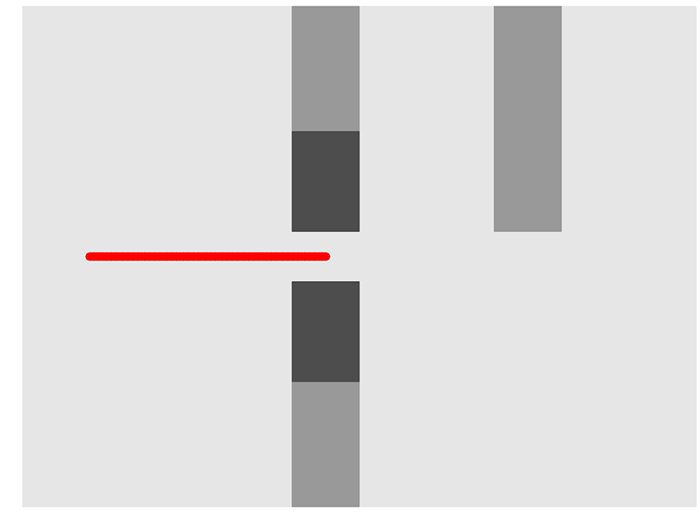} 
}
\subfigure[] { \label{fig:intro4}
\includegraphics[width=0.31\columnwidth]{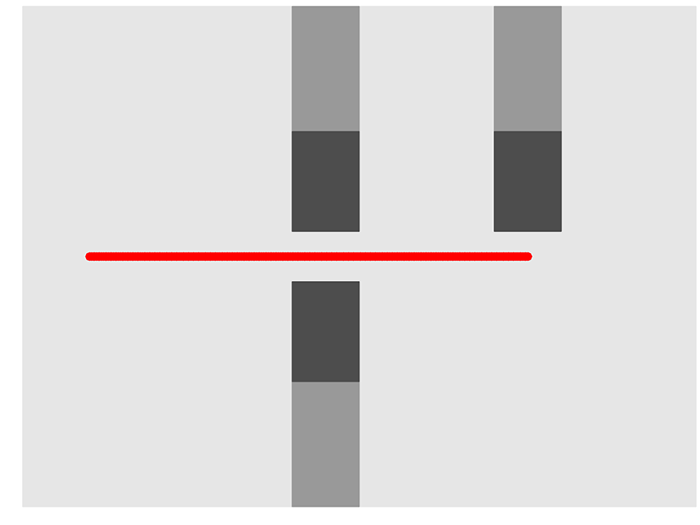} 
}
\subfigure[] { \label{fig:intro5}
\includegraphics[width=0.31\columnwidth]{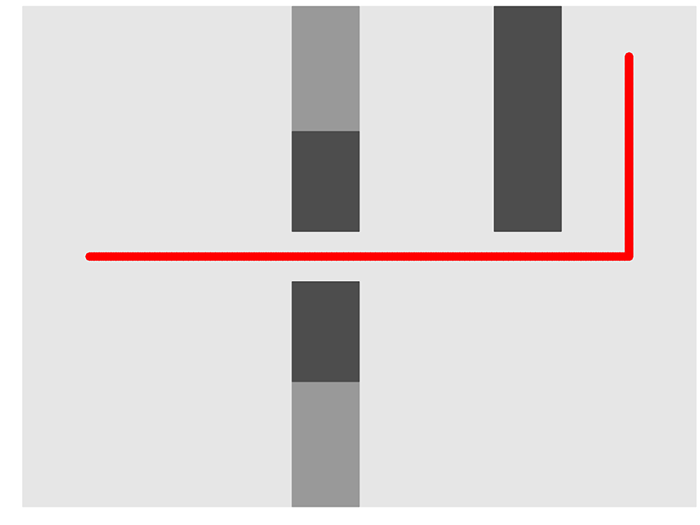} 
}
\caption{\small Environment, Moving and Obstacle Free Tube: The environment obstacles are the light and dark gray tubes (light as undetected and dark as detected). In (a), the red diamond and circle are the start and target configurations of the robot and the shaded tubular region is the obstacle-free region $\mathcal{T}$. (b)-(e) are the robot moving process along a certain path in chronological order.}
\label{fig:intro}
\end{figure}

For the convenience of discussion, we assume there exists at least a feasible path connecting the initial and target configurations, and this feasible path is contained in a tubular obstacle-free region $\mathcal{T}$ with radius $L$ as shown by the shadow part in the first picture in Figure \ref{fig:intro}. This assumption is a technique requirement that is used for the proof of the convergence and can be rewritten as the following equation
\begin{align}
\sup_{\gamma\in\Gamma}\inf_{t\in[0,T]}\sup_{r\geq0}\left\lbrace r:B(\gamma(t),r)\cap\mathcal{O}=\emptyset\right\rbrace=L>0,
\label{equ:feasibleregion}
\end{align}
where $\mathcal{O}=\left\lbrace x\in\Omega:\psi_i(x)<0\text{ for some }i\in\{1,\cdots,k_1\}\text{ or }\phi_i(x)<0\text{ for some }i\in\{1,\cdots,k_2\}\right\rbrace$ is an open set, and $B(x,r)=\{y\in\Omega:\|x-y\|<r\}$ is also open. We denote $S(x,r)=\{y\in\Omega:\|x-y\|=r\}$ as the boundary of $B(x,r)$, and $\partial\mathcal{O}$ as the boundaries of $\mathcal{O}$ separating the constrained regions from the feasible regions. 

Let us define the set of all possible paths from $x_s$ to $x_t$ in the full time interval $[0,T]$ as
\begin{align}
\Gamma=\left\lbrace\gamma:\gamma(0)=x_s,\gamma(t)=x_t,\forall t\geq T_0,\text{ for some }T_0\leq T,\gamma\cap\mathcal{O}=\emptyset\right\rbrace
\label{equ:feasiblepath}
\end{align}
Then our dynamical path planning algorithm in the unknown environment is given in Algorithm \ref{alg:1}:\\
\begin{algorithm}[H]
\DontPrintSemicolon
\SetAlgoLined
\caption{Path Planning in Unknown Environment}
\label{alg:1}
\KwData{initial configuration $x_s$, target configuration $x_t$, initially known constraints $\mathcal{O}_0$}
Current configuration $x_c=x_s$\;
Current known constraints $\mathcal{O}_c=\mathcal{O}_0$\;
\While{$x_c\neq x_t$}{
\textbf{Graph Generating:} Generate a connected graph $G$ containing $x_c,x_t$ with all edges and vertices not in $\mathcal{O}_c$\;
\textbf{Path Finding:} Find a (shortest) path $\gamma$ on $G$ from $x_c$ to $x_t$\;
\textbf{Environment Updating:} Moving along $\gamma$ while updating $\mathcal{O}_c$, if $\gamma$ is blocked by $\mathcal{O}_c$, stop at $x$ near the block point, otherwise let $x_c=x_t$\;
}
\end{algorithm}
In the remaining part of this section, we discuss, assisted with examples, the three major steps (Step 4, 5, and 6 in Algoirhtm \ref{alg:1}) in details.

\subsection{Graph generating} \label{graph_generating}
The first step is to generate a graph $G=(V,E)$, where $V$ is the vertex set and $E$ is the edge set, connecting the current configuration $x_c$ ($x_c=x_s$ for the first graph generation) and the target $x_t$ with currently known environment. The vertices are configurations in $\Omega$ while the edge $(u,v)$ linking $u,v\in V$ is the straight line segment between $u$ and $v$. Meanwhile, we would like to create the graph satisfying two properties: 1) the graph does not violate any known constraints; 2) the graph cannot contain too many vertices due to the computation complexity concern in the high dimensional cases. To achieve these goals, we introduce a convex potential function $p(x)$, admitting a unique global minimizer $x_t$, to help us choose the vertices. We select an $n$-dimensional orthonormal basis $N$ (here $n$ is the dimension of $\Omega$) to determine the directions which are used to add new vertices to $V$. For simplicity, we take $p(x)=\|x-x_t\|$, the distance to the target, as the potential, and the standard coordinate axes $N = \{e_j\}_{j=1}^{n}$ as the orthonormal basis in this paper.

At the first generating step, we let $V=\{x_c\}$ and $E=\emptyset$. In each step afterward, a vertex $v\in V$ with the lowest potential is chosen. We pick $2n$ new points $\{v_i\}_{i=1}^{2n}$ along the orthonormal basis $N$ originated at $v$, with distance $l$ to $v$, and use them as the candidates to expand $V$ (first figure in Figure \ref{fig:gg1}). Before adding those points into the vertex set, we first delete all candidates that violate the currently known constraints ($\phi_k(v_i)<0$ or $\widetilde{\psi}_j(v_i,\gamma,T_0)<0$ for some $k\in\{1,\cdots,k_1\}$ and $j\in\{1,\cdots,k_2\}$, $\gamma$ is the previous trajectory of the robots). For example, the robot shown in the left plot in Figure \ref{fig:gg1} stops at the red diamond position and generates four points around it. Among them, the point in the obstacle is removed (right picture in Figure \ref{fig:gg1}). Next, we delete vertices whose edges violate the constraints as shown in Figure \ref{fig:gg2}. In this case, there exists a point $x\in(v_i,v)$ such that $x$ is not in the feasible region. In addition, to avoid repeating vertices, we remove those already included in $V$ from the candidate list, as shown in Figure \ref{fig:gg3}. After these deleting steps, we add all remaining candidates, and their associated edges, to $V$ and $E$ respectively. This process is repeated until the target $x_t$ is within a small neighborhood of a vertex in $V$. For example, the final graph after several iterations is plotted in the right figure of Figure \ref{fig:gg3}.

To summarize, given the previous trajectory of robots $\gamma$ ($\gamma=x_s$ as default), we let the current constraints be
\[
\mathcal{O}_c=\left\lbrace x\in\Omega:\widetilde{\psi}_i(x,T_0,\gamma)<0\text{ for some }i\in\{1,\cdots,k_2\}\text{ or }\phi_i(x)<0\text{ for some }i\in\{1,\cdots,k_1\}\right\rbrace.
\]
The graph generating procedure can be written in the following algorithm (Algorithm \ref{alg:2}), and we define $x$ is the ancestor of $y$ if when $x$ is picked to generate nodes as the lowest potential node, $y$ is added to the vertex set as newly generated node.

\begin{algorithm}[H]
\DontPrintSemicolon
\SetAlgoLined
\caption{Graph Generation}
\label{alg:2}
\KwData{The starting configuration $x_c$, target configuration $x_t$, the potential function $p$, currently known environment $\mathcal{O}_c$, graph generating radius $l$ and a set of orthonormal basis $N$}
\KwResult{G=(V,E,p)}
$V=\{x_c\},Q=V,E=\emptyset$\;
\While{$t\not\in V$}{
	$point\_add=False$\;
	\While{not $point\_add$}{
		$v=\arg\min_{x\in V}p(x)$\;
		\uIf{$p(v)<+\infty$}{
			$K=\{q:q=v\pm l\times y,y\in N,(v,q)\cap\mathcal{O}_c=\emptyset,q\notin\mathcal{O}_c\}$\;
		    $K=K\backslash V$\;
		    $V=V\cup K$\;
		    $E=E\cup\{(v,q):q\in K\}$\;
		    \If{$K\neq\emptyset$}{
		    	$point\_add=True$\;
		    	$p(v)=+\infty$\;
		    }
		    \For{$q\in K$}{
		        \If{$\|q-x_t\|\leq L$ and $(q,x_t)\cap\mathcal{O}_c=\emptyset$}{
		            $V=V\cup\{x_t\}$\;
		            $E=E\cup\{(q,x_t)\}$\;
		        }
		    }
		}\Else{
		\Return{$G=\emptyset$\;}
		}
	    
	}
}
\Return{$G=(V,E,p)$\;}
\end{algorithm}

\begin{remark} The choice of the generating radius $l$ can be arbitrary, although $L$ and $l$ must satisfy an inequality to have the convergence guarantee theoretically (as is shown in Section \ref{theories}). Larger $l$ leads to fewer vertices in $V$ while smaller $l$ giving a finer search in $\Omega$. For simplicity, we treat those obstacles with distance less than $l$ to be a single obstacle by ignoring the gaps among them in our theoretical analysis. In practice, as shown in our experiments, the graph generation can still create nodes passing through the gap between obstacles with distance less than $L$ or even $l$. \end{remark}

\begin{figure}\centering
\includegraphics{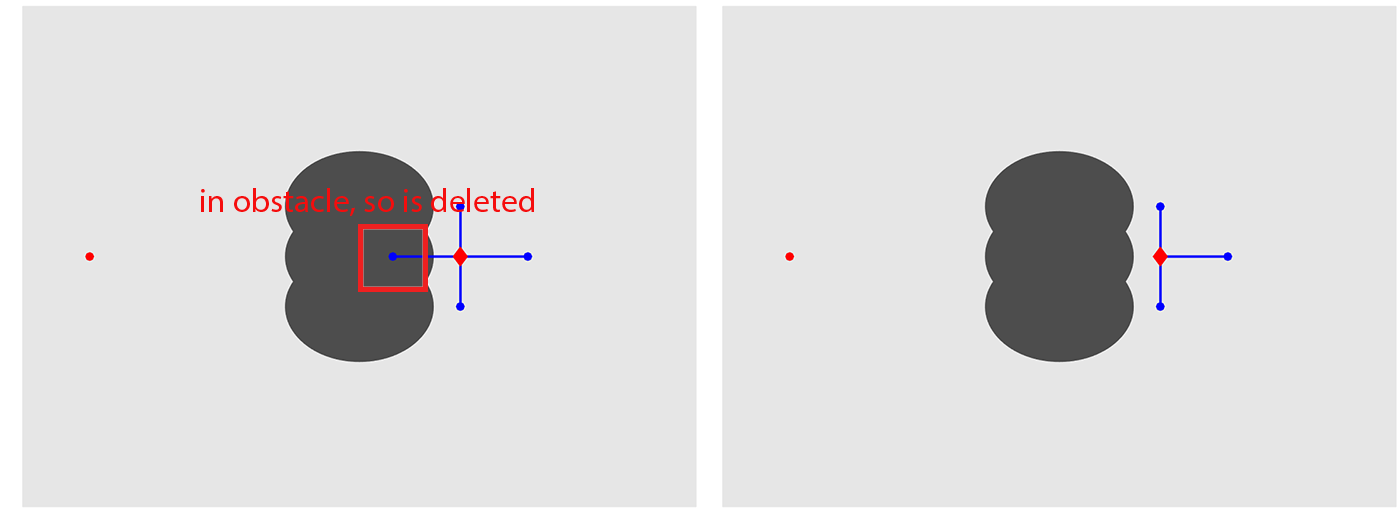}
\caption{\small graph generating steps: delete nodes in obstacles}
\label{fig:gg1}
\end{figure}
\begin{figure}\centering
\includegraphics{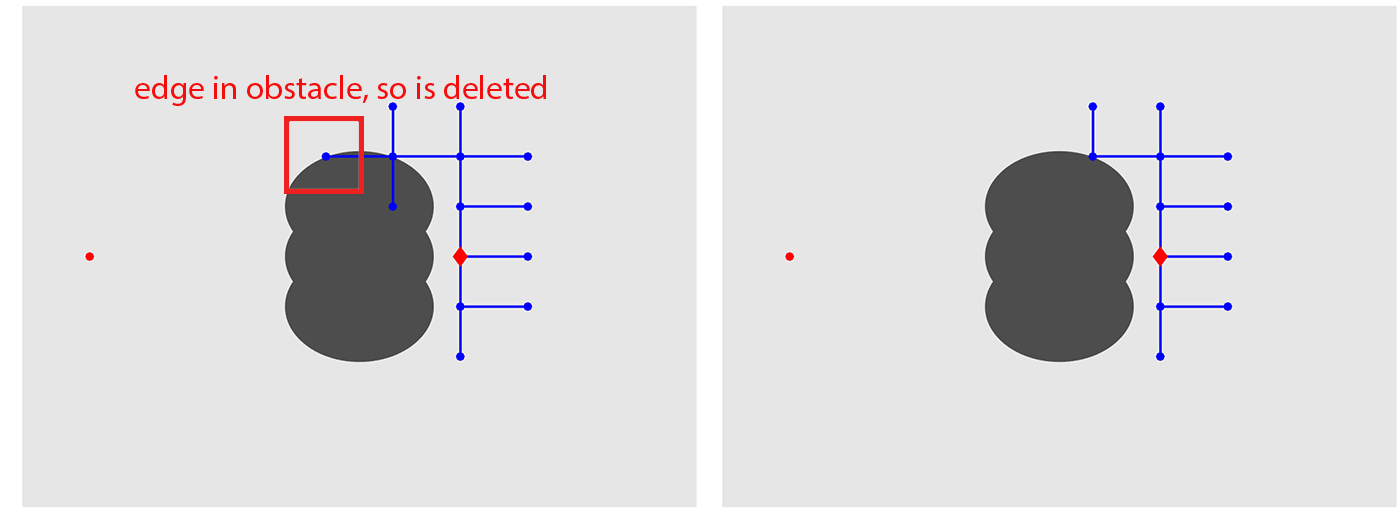}
\caption{\small graph generating steps: delete nodes cannot be linked to the base node}
\label{fig:gg2}
\end{figure}
\begin{figure}\centering
\includegraphics[width=1.0\columnwidth]{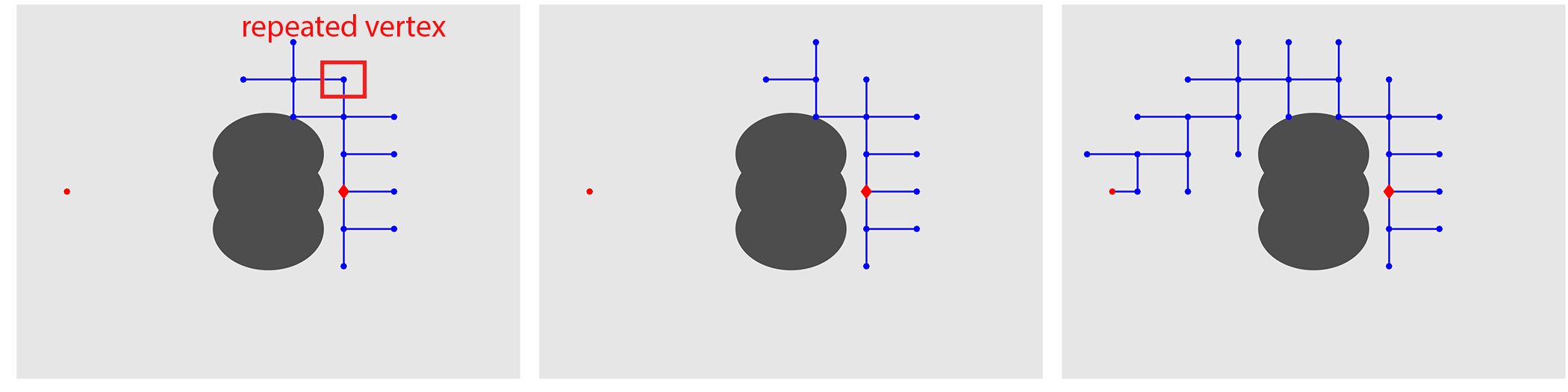}
\caption{\small graph generating steps: delete repeated nodes (left two) and final graph (right one)}
\label{fig:gg3}
\end{figure}

\subsection{Path finding} \label{path_finding}
After generating the graph $G=(V,E,p)$, the next step is to find a feasible path moving from the current configuration to the target using only vertices and edges on the graph. Our goal is to minimize the total travel distance. The graph generated by Algorithm \ref{alg:2} has the following property:
\begin{proposition}
There exists a unique path from the current configuration $x_c$ to the target $x_f$ over the generated graph $G$. If the path is denoted by $\{x_i\}_{i=1}^q\subset V$ with
\begin{equation*}
x_c=x_1\rightarrow x_2\rightarrow\cdots\rightarrow x_q=x_f,
\end{equation*}
$x_i$ is the ancestor of $x_{i+1}$.
\end{proposition}
\begin{proof}
The existence of a path from the current configuration to the target is provided by the construction of the graph. We note that if two nodes share an edge, one of them is the ancestor of the other, which is also determined by the graph construction algorithm. Clearly, we have that $x_c$ is the ancestor of $x_2$. By induction, if we assume that $x_{i}$ is the ancestor of $x_{i+1}$, we claim that $x_{i+1}$ is the ancestor of $x_{i+2}$. Otherwise, it implies that $x_{i+2}$ must be the ancestor of $x_{i+1}$, which means that $x_{i+i}$ has two ancestors. This is a contradiction with the graph generation strategy: we delete all candidate nodes that are already generated in previous steps. 

For the uniqueness of the path, we first notice that the algorithm stops once an edge is linked to the target $x_f$, from which we conclude that $x_f$ has unique ancestor. By Algorithm \ref{alg:2}, every node except $x_c$ has a unique ancestor. If we assume that there are two paths $\gamma$ and $\xi$, denoted by $\{y_i\}_{i=1}^{m_1}$ and $\{z_i\}_{i=1}^{m_2}$ respectively, we must have $y_{1}=z_{1}=x_c$ and $y_{m_1}=z_{m_2}=x_f$. By the uniqueness of the ancestor for each node, we must have $y_{m_1-1}=z_{m_2-1}$. By induction, we have $m_1=m_2$ and $y_i=z_i$. Thus, $\gamma=\xi$ and the uniqueness is proven.
\end{proof}
By our graph generation algorithm, if there is an edge between two nodes, one of them must be the unique ancestor of the other. This suggests a simple strategy to identify the path: from the target configuration, we simply back trace the ancestor of each node in the path until reaching starting configuration $x_c$.

Other algorithms can be applied to find the path as well. For example, we can define the distance of the edge $e_{ij}$ linking vertices $v_i,v_j$ as 
\begin{align*}
k_{ij}=len((i,j))=\|v_i-v_j\|.
\end{align*}
Then the well-known Dijkstra method, or its improvements, can be used to obtain the path with computational complexity $O(|E|+|V|\log|V|)$ where $|V|$ is the number of vertices and $|E|$ is the number of edges \cite{dijkstra1959note}.

Another way is to assign each edge distance $1$ which is equivalently to introduce the modified adjacency matrix $K=(k_{ij})$ on the graph $G$, where
\[
k_{ij}=\left\lbrace\begin{array}{ll}
1&\text{if }(i,j)\in E\\
\infty&\text{otherwise}.
\end{array}\right.
\]
Then the Breadth First Search (BFS) can be used to find the path with the complexity $O(|E|+|V|)$ \cite{zuse1972plankalkul,moore1959shortest}, which is faster than Dijkstra. Other graph-based path planning algorithms, such as $A^*$, $D^*$ or $D^*$ lite, can be used too.

It is worth mentioning that if we assume the path has $\Lambda$ nodes, the suggested back-tracing approach is of complexity $O(\Lambda)$. While the generated graph has at least $O(n\Lambda)$ nodes. Obviously the complexities of BFS and Dijkstra methods are higher than our back-tracing strategy.

\subsection{Environment updating} \label{environment_updating}
While robots move along a path $\gamma_i$ in the configuration space, the knowledge of constraints is updated at the same time by (\ref{equ:envupdate}). We let
\[
\gamma_i:[0,T_0]\rightarrow\Omega
\]
be the current path, and if a point on the path intersects the boundary of the constrained region, the motion stops at a point before arriving the intersection. 

To be more precise, let us denote the environment update at each time step as
\[
\mathcal{O}_c^t=\mathcal{O}_c\cup\left\lbrace x\in\Omega:\widetilde{\psi}_j(x,t,\gamma_i)<0\text{ for some }j\in\{1,\cdots,k_2\}\right\rbrace.
\]
If the path is found activating constraints while moving at time $T_s$, i.e. $\gamma_i\cap\mathcal{O}_c^{T_s}\neq\emptyset$, we define
\[
T_b=\inf\{t:\gamma_i(t)\in\mathcal{O}_c^{T_s}\}
\]
as the first intersection time. Then $\gamma(T_b)$ must be on the boundary of $\mathcal{O}_c^{T_s}$, i.e. $\gamma(T_b)\in\partial\mathcal{O}_c^{T_s}$. When this happens, we can always pick a stopping time $T_i\leq T_b$ such that the distance from $\gamma(T_i)$ to the nearest obstacle is smaller than the detection radius $R$. Afterwards, we update $\mathcal{O}_c=\mathcal{O}_c^{T_i}$, assign the initial configuration as $x_c=\gamma(T_i)$ and go back to the graph generating step. Each time a new path $\gamma_i$ is produced when the current path is blocked. We collect all paths produced in Algorithm \ref{alg:1} as $\{\gamma_i\}_{i=1}^m$, and their stopping time set as $\{T_i\}_{i=1}^m$. From our choices of stopping time, we can require that there exists a positive constant $q$ satisfying $q<R$, and for all $\epsilon>q$, $B(\gamma_i(T_i),\epsilon)$ has non-empty intersection with $\mathcal{O}_c$ for every $i=1,\cdots,m$. Such selected stopping time set satisfies the following property
\begin{align}
\sup_i\inf_\epsilon\left\lbrace\epsilon:B(\gamma_i(T_i),\epsilon)\cap\mathcal{O}_c^{T_i}\neq\emptyset\right\rbrace=q<R,
\label{equ:stop}
\end{align}
in which the detectable region at configuration $x$, using (\ref{equ:envupdate}), is defined as a closed set by
\begin{align*}
\bar{B}(x,R)=\{y\in\Omega:d(x,y)\leq R\}
\end{align*}
We emphasis that $q$ can be selected uniformly. For example, we can simply let robots stop at a position that has a distance of $R/2$ to the obstacles each time when the path is blocked. In this setup, $q=R/2<R$. In general, we can select different stop positions. The finite-step convergence property, presented in the next Section, is guaranteed as long as \eqref{equ:stop} is satisfied.

\subsection{Convergence and complexity} \label{theories}
The proposed algorithms terminate in finite steps with guaranteed convergence, which is stated in the following main theorems.

\begin{theorem} \label{thm1} Assuming that (\ref{equ:feasibleregion}) is true and 
\[
l<\frac{2L}{\sqrt{n}},
\]
where $n$ is the dimension of $\Omega$, the graph generation algorithm (Algorithm \ref{alg:2}) stops in finite steps. That is,
the loop in the algorithm terminates in finite iterations, the generated graph $G=(V,E,K,p)$ is connected and has a finite number of vertices $|V|<\infty$. Furthermore, $x_s, x_t\in V$ if $\Gamma\neq\emptyset$.
\end{theorem}

\begin{theorem} \label{thm2} Let $\{\gamma_i\}_{i=1}^m$ be the paths produced by Algorithm \ref{alg:1} with $\{T_i\}_{i=1}^m$ being the stopping time set. If the assumptions in Theorem \ref{thm1} and (\ref{equ:stop}) hold, then $m<\infty$.
\end{theorem}

Theorem \ref{thm1} shows that, given the currently known environment, the graph generating procedure stops in finite steps. Theorem \ref{thm2} tells that our algorithm breaks the loop in Algorithm \ref{alg:1} in finite steps. The two theorems together ensure that Algorithm \ref{alg:1} is convergent in finite steps and guarantees a feasible path with the condition (\ref{equ:feasibleregion}). Therefore, the algorithm is complete. We leave the proofs of both theorems in Section \ref{convergence}.

Furthermore, if the configuration space $\Omega$ is of dimension $n$, there are at most $2n$ new points generated at each step in Algorithm \ref{alg:2}, hence the growth rate for the size of the graph $V$ is $O(n)$ at each iteration. The complexity of the Updating Environment step relies on the techniques used to detect the environment, so we do not consider it here. Overall, the proposed algorithms are scalable to high dimensional problems, because the growth of the graph is controlled linearly with respect to the dimension $n$ and it stops in finite steps. This feature is illustrated by our numerical examples presented in the next section.

\section{Numerical Examples} \label{example}
We set the working space to be $[0,1]\times[0,1]$ in all examples and denote the graph generating radius as $l$. In this section, we show various low-dimensional experiments to give a basic impression on how the algorithm works, followed by several high dimensional cases with different environments. In all examples, the start configurations are always marked as red diamonds while the targets are the red circles. The potential function is taken as the Euclidean distance from any point $x$ to the target $x_t$, i.e. $p(x)=\|x-x_t\|$ where $x,x_t\in\Omega\subset\mathbb{R}^n$.

\subsection{Low dimensional cases}
The first example is one robot moving in an unknown environment (Figure \ref{fig:g1}). The configurations are the physical locations of the robot, so this is a two-dimensional problem. We take $l=0.03$ in our graph generating algorithm. Initially, the robot is at the top right corner. It only has the knowledge of a few nearby obstacles at the beginning, while other obstacles are not known. Hence, the graph expands towards the target greedily until reaching the destination as shown in Figure \ref{fig:g1}(a). The first path is found by BFS on this graph, shown in Figure \ref{fig:p1}(a). However, while moving, the robot detects that the path is blocked. It stops before reaching the obstacle boundary and starts a new round of graph generating, path finding and environment updating steps. During the process, the robot generates several graphs (Figure \ref{fig:g1}(b,c,d)) and updates the environment while moving along the corresponding paths as shown in Figure \ref{fig:p1}(b,c,d), all of which fail to reach the destination. In the end, it generates a graph (Figure \ref{fig:g1}(e)) and finds a path (Figure \ref{fig:p1}(e)) to the target. The complete path from initial to target is provided in Figure \ref{fig:p1}(f).

The set-ups for the next example are all the same as the previous one except the initial configuration. The generated graphs are depicted in Figure \ref{fig:g2}(a-e) in time order and the corresponding paths are in Figure \ref{fig:p2}(a-e) while the complete path is shown in Figure \ref{fig:p2}(f). Despite of the difference in the initial configurations, the algorithm gives similar paths (Figure \ref{fig:p1}(f) and \ref{fig:p2}(f)). In the next experiment, we keep the settings used in the second example, but enlarge the generating radius $l$ from $0.03$ to $0.05$. By doing so, the robot can no longer move into the central box from the top left corner as it does in the first two examples (Figure \ref{fig:p1}(e) and \ref{fig:p2}(e)). Instead, it moves down and finds a different way to the destination. This path reaches higher potential area than the previous paths. The graphs for this example are depicted in Figure \ref{fig:g3} and paths are in Figure \ref{fig:p3} respectively. 

To conclude the lower dimensional cases, we display the graphs and paths produced by Algorithm \ref{alg:1} for a different environment in Figure \ref{fig:g4} and \ref{fig:p4}. Similar behaviors can be observed in those pictures.

\begin{figure} \centering
\subfigure[] { \label{fig:g11}
\includegraphics[width=0.31\columnwidth]{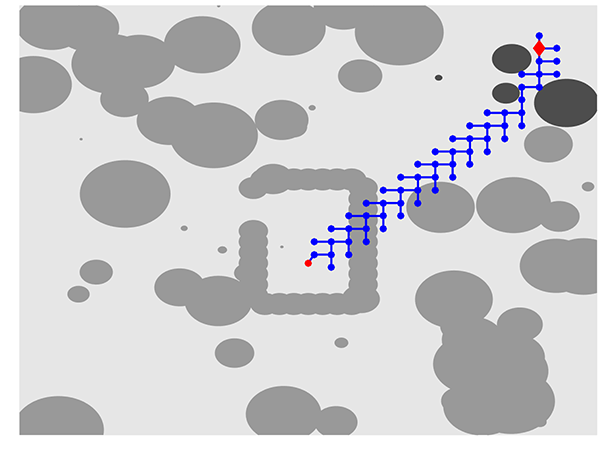} 
}
\subfigure[] { \label{fig:g12}
\includegraphics[width=0.31\columnwidth]{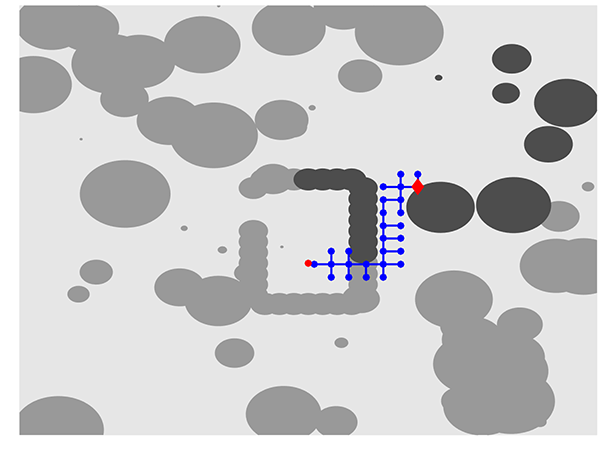} 
}
\subfigure[] { \label{fig:g13}
\includegraphics[width=0.31\columnwidth]{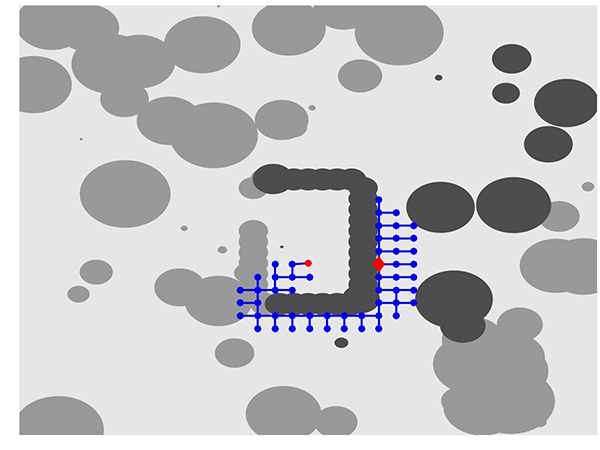} 
}
\subfigure[] { \label{fig:g14}
\includegraphics[width=0.31\columnwidth]{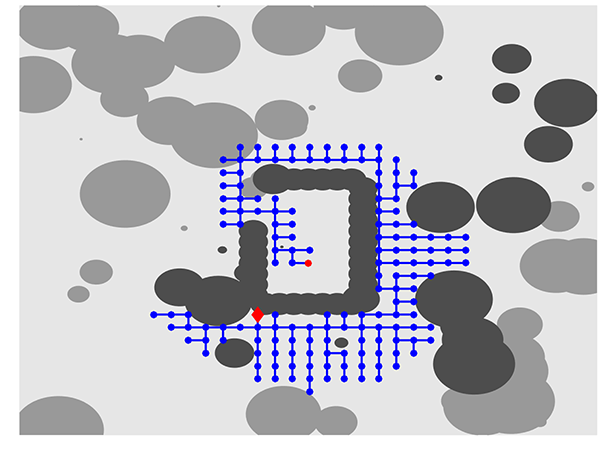} 
}
\subfigure[] { \label{fig:g15}
\includegraphics[width=0.31\columnwidth]{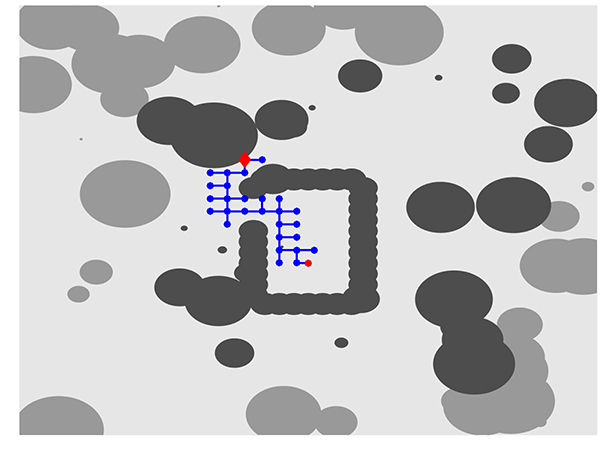} 
}
\caption{\small The graph produced in the one robot case with generating radius $l=0.03$ and light (dark) gray the undetected (detected) obstacles. The graph expands greedily towards the target. If obstacles are on the greedy direction, it searches around the obstacles and generates new nodes with potential as low as possible. (a)-(d) are graphs that cross undetected obstacles so the robot stops while moving on those graphs. With enough environment knowledge, (e) is a graph containing a true feasible path from the current initial configuration and the target.}
\label{fig:g1}
\end{figure}

\begin{figure} \centering
\subfigure[] { \label{fig:p11}
\includegraphics[width=0.31\columnwidth]{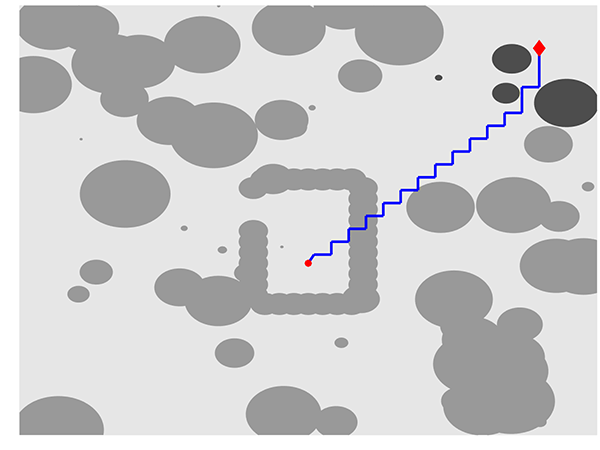} 
}
\subfigure[] { \label{fig:p12}
\includegraphics[width=0.31\columnwidth]{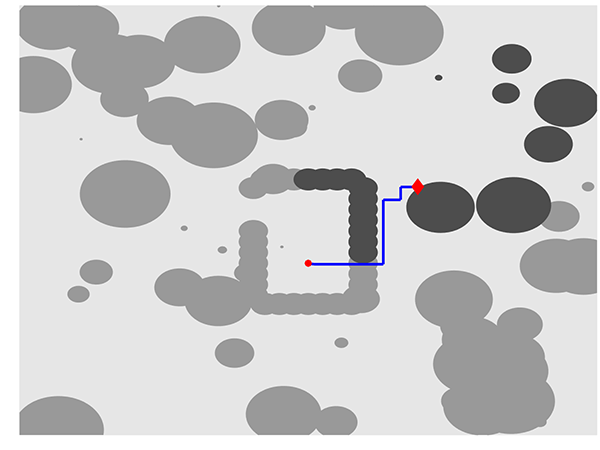} 
}
\subfigure[] { \label{fig:p13}
\includegraphics[width=0.31\columnwidth]{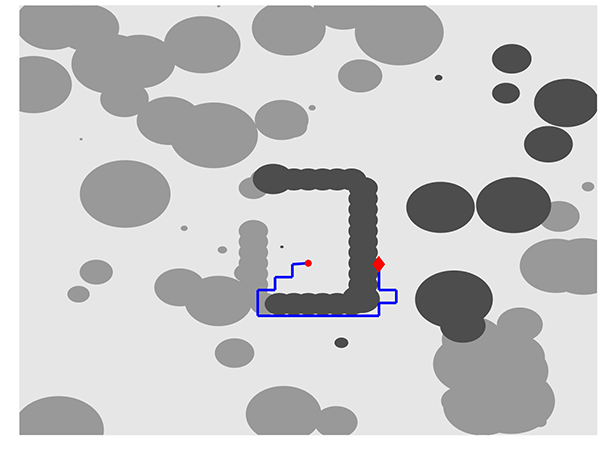} 
}
\subfigure[] { \label{fig:p14}
\includegraphics[width=0.31\columnwidth]{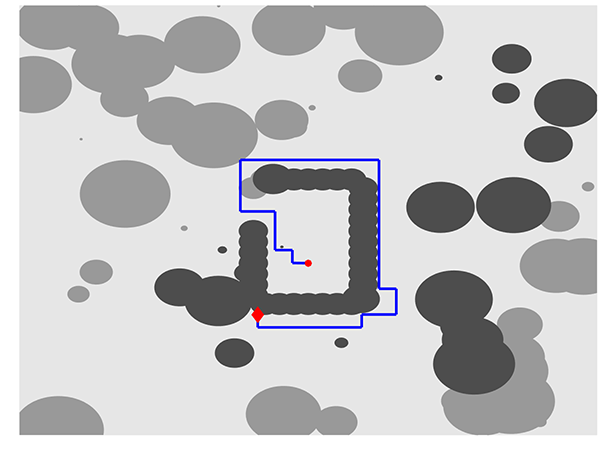} 
}
\subfigure[] { \label{fig:p15}
\includegraphics[width=0.31\columnwidth]{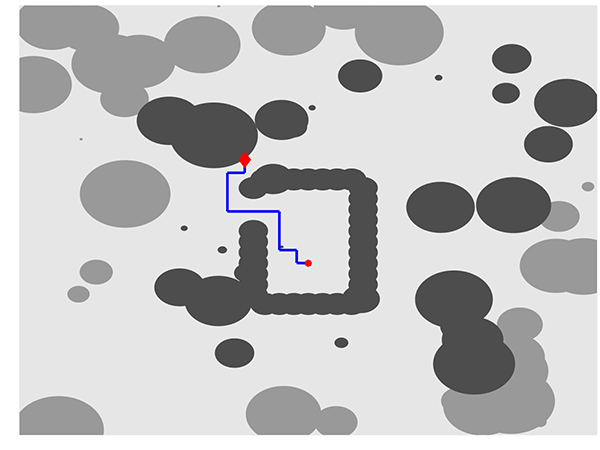} 
}
\subfigure[] { \label{fig:p16}
\includegraphics[width=0.31\columnwidth]{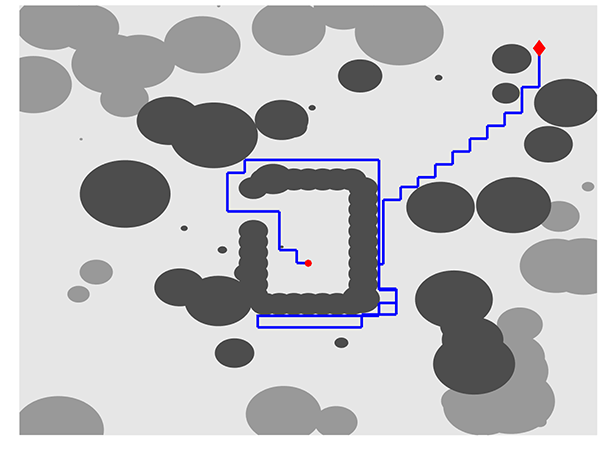} 
}
\caption{\small The paths calculated based on the results in Figure \ref{fig:g1}. (a)-(d) are middle steps that the robot stops because of the newly detected obstacles while moving and (e) is the path on which the robot get to the target. (f) gives the complete path of the robot.}
\label{fig:p1}
\end{figure}

\begin{figure} \centering
\subfigure[] { \label{fig:g21}
\includegraphics[width=0.31\columnwidth]{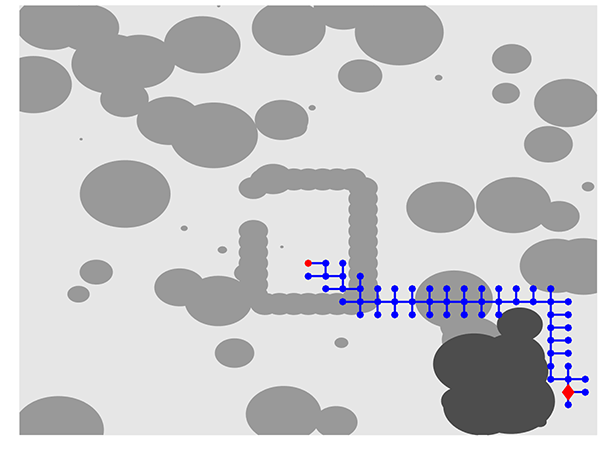} 
}
\subfigure[] { \label{fig:g22}
\includegraphics[width=0.31\columnwidth]{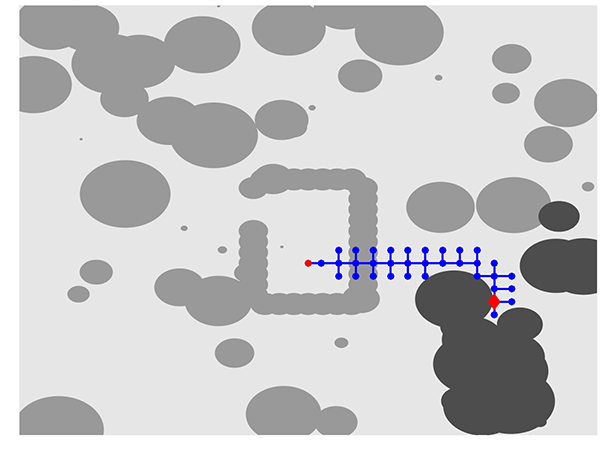} 
}
\subfigure[] { \label{fig:g23}
\includegraphics[width=0.31\columnwidth]{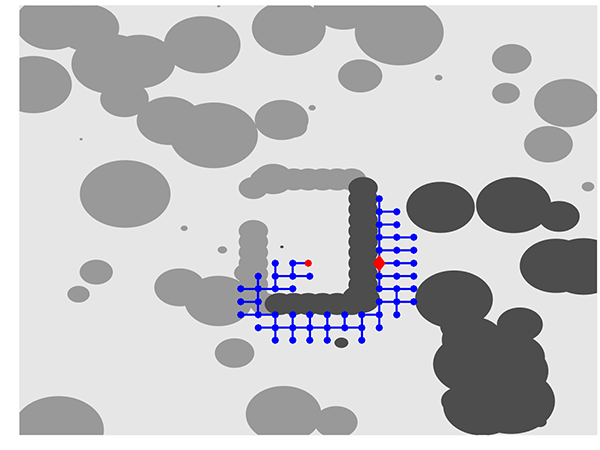} 
}
\subfigure[] { \label{fig:g24}
\includegraphics[width=0.31\columnwidth]{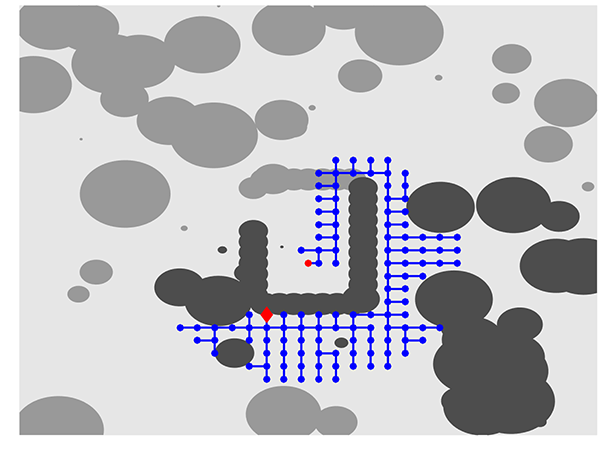} 
}
\subfigure[] { \label{fig:g25}
\includegraphics[width=0.31\columnwidth]{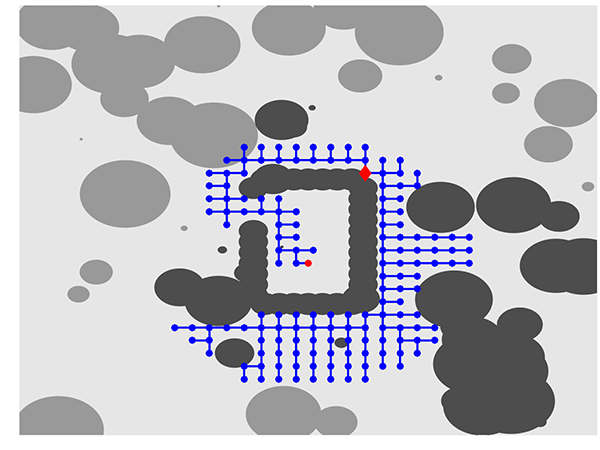} 
}
\caption{\small The graph produced under the same environment and target configuration as Figure \ref{fig:g1} but with a different initial configuration located at top right corner. The graphs search almost the same region around the central box as those in Figure \ref{fig:g1} except (a).}
\label{fig:g2}
\end{figure}

\begin{figure} \centering
\subfigure[] { \label{fig:p21}
\includegraphics[width=0.31\columnwidth]{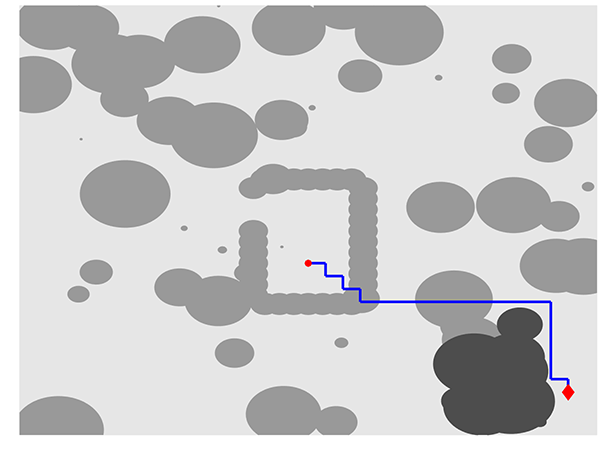} 
}
\subfigure[] { \label{fig:p22}
\includegraphics[width=0.31\columnwidth]{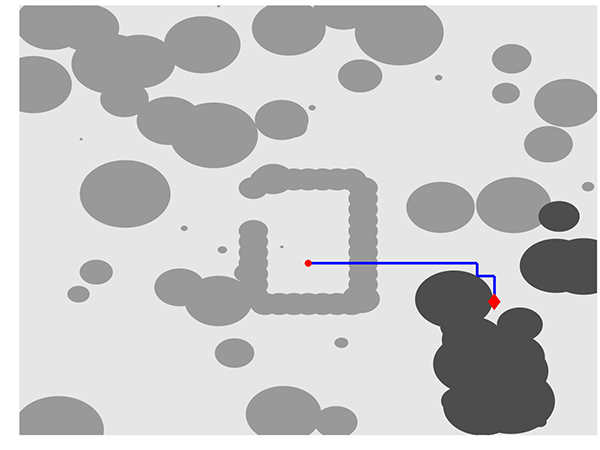} 
}
\subfigure[] { \label{fig:p23}
\includegraphics[width=0.31\columnwidth]{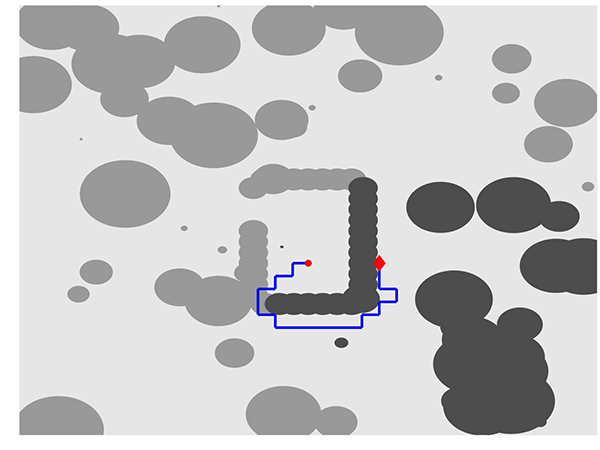} 
}
\subfigure[] { \label{fig:p24}
\includegraphics[width=0.31\columnwidth]{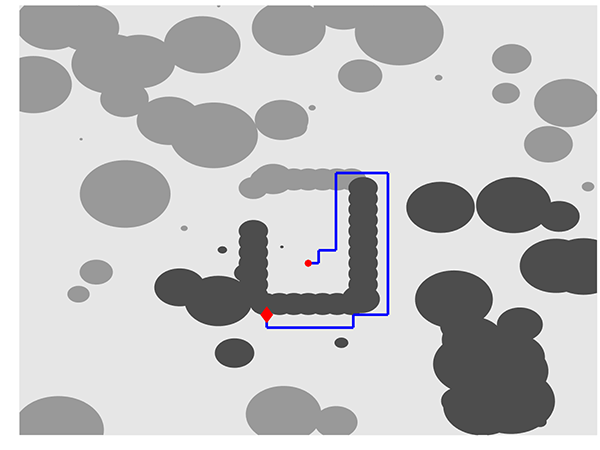} 
}
\subfigure[] { \label{fig:p25}
\includegraphics[width=0.31\columnwidth]{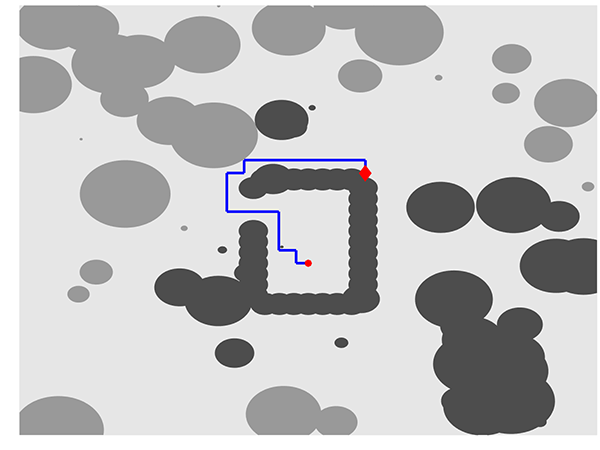} 
}
\subfigure[] { \label{fig:p26}
\includegraphics[width=0.31\columnwidth]{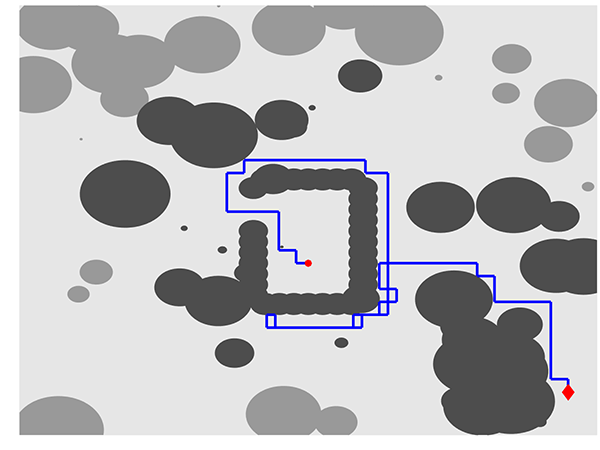} 
}
\caption{\small The paths calculated based on the results in Figure \ref{fig:g2}. (a)-(d) are middle steps that the robot stops because of the newly detected obstacles while moving and (e) is the path on which the robot get to the target. (f) gives the complete path of the robot. The paths are similar to those in Figure \ref{fig:p1} in spite of different initial configuration.}
\label{fig:p2}
\end{figure}

\begin{figure} \centering
\subfigure[] { \label{fig:g31}
\includegraphics[width=0.31\columnwidth]{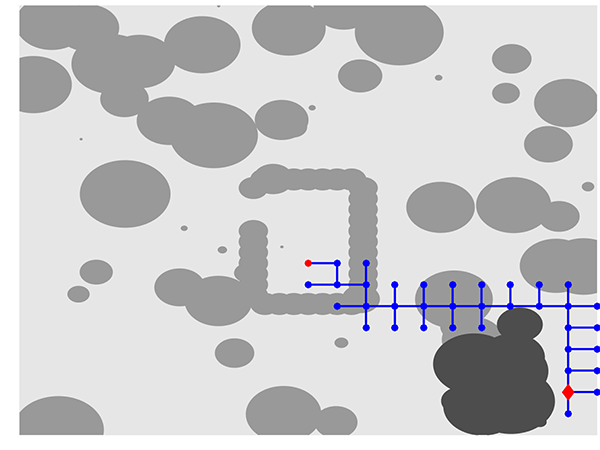} 
}
\subfigure[] { \label{fig:g32}
\includegraphics[width=0.31\columnwidth]{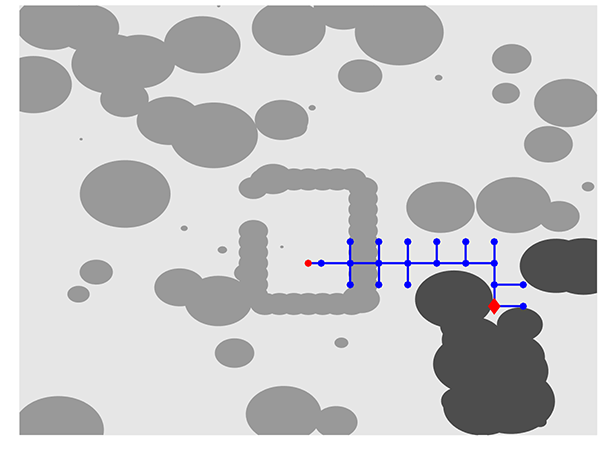} 
}
\subfigure[] { \label{fig:g33}
\includegraphics[width=0.31\columnwidth]{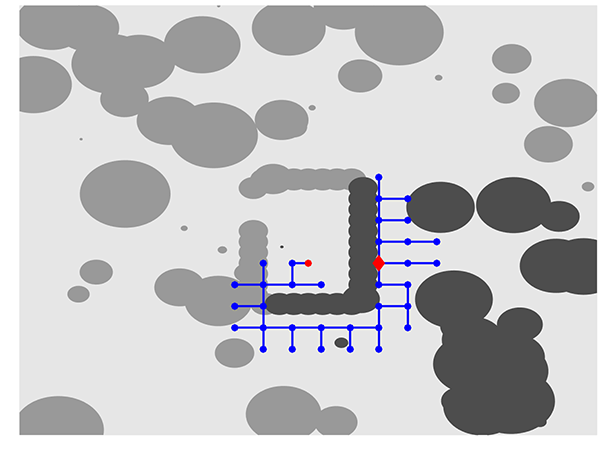} 
}
\subfigure[] { \label{fig:g34}
\includegraphics[width=0.31\columnwidth]{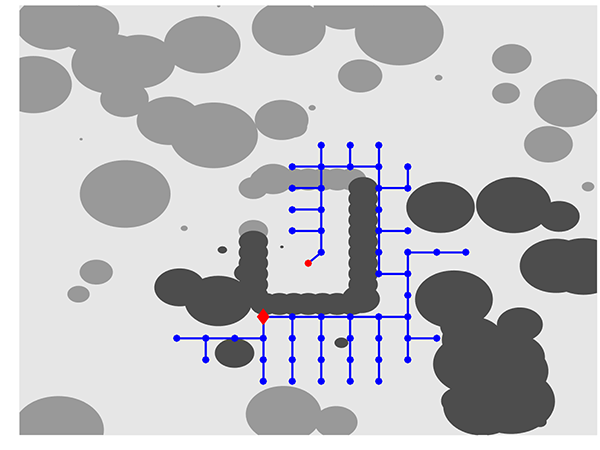} 
}
\subfigure[] { \label{fig:g35}
\includegraphics[width=0.31\columnwidth]{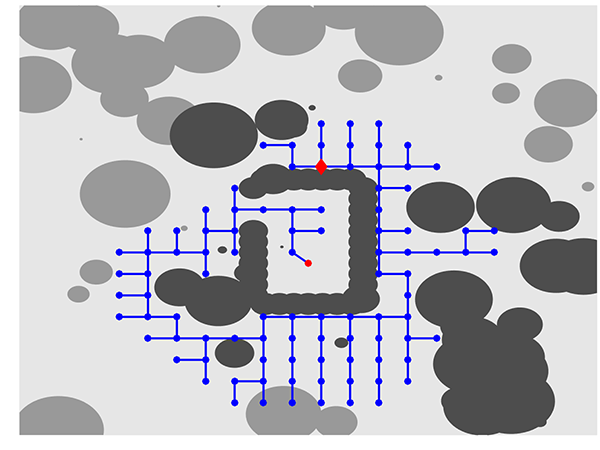} 
}
\caption{\small The graph produced under the same environment, initial and target configuration as Figure \ref{fig:g2} but with a larger generating radius $l=0.05$. The robot can no longer move through the original path, so to get to the target, it finds a different path that contains points with higher potential.}
\label{fig:g3}
\end{figure}

\begin{figure} \centering
\subfigure[] { \label{fig:p31}
\includegraphics[width=0.31\columnwidth]{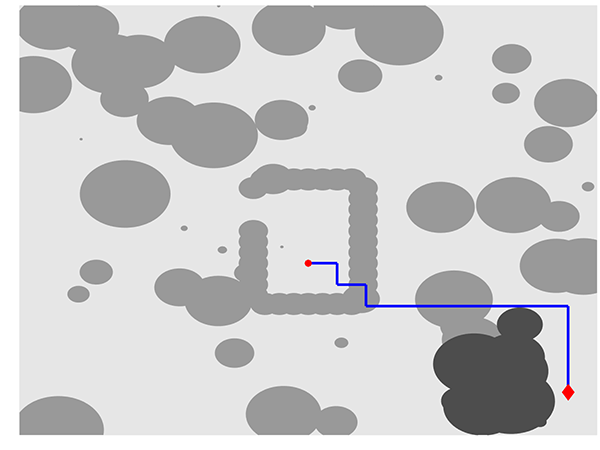} 
}
\subfigure[] { \label{fig:p32}
\includegraphics[width=0.31\columnwidth]{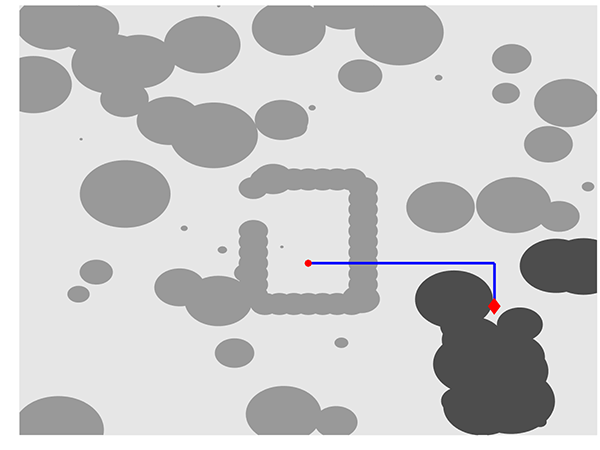} 
}
\subfigure[] { \label{fig:p33}
\includegraphics[width=0.31\columnwidth]{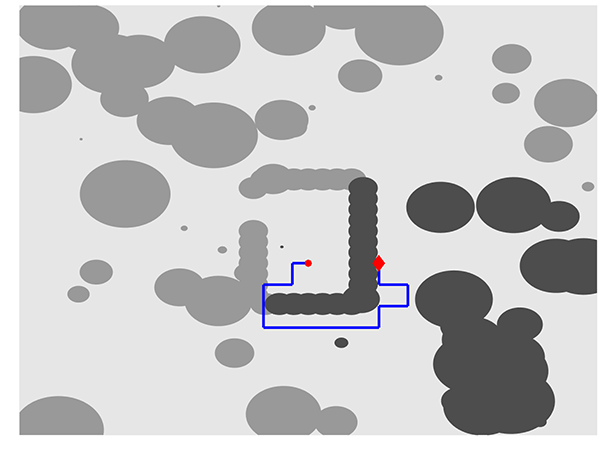} 
}
\subfigure[] { \label{fig:p34}
\includegraphics[width=0.31\columnwidth]{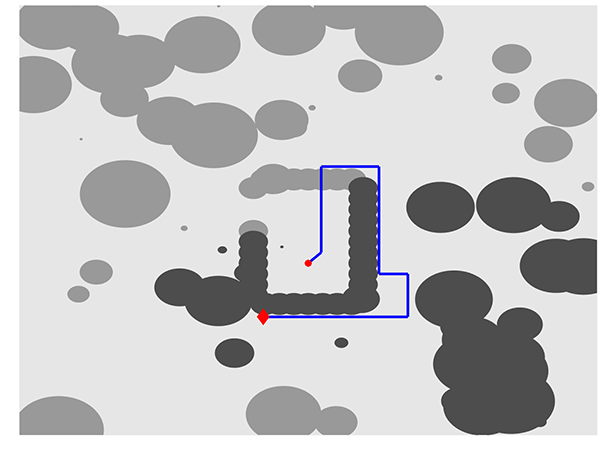} 
}
\subfigure[] { \label{fig:p35}
\includegraphics[width=0.31\columnwidth]{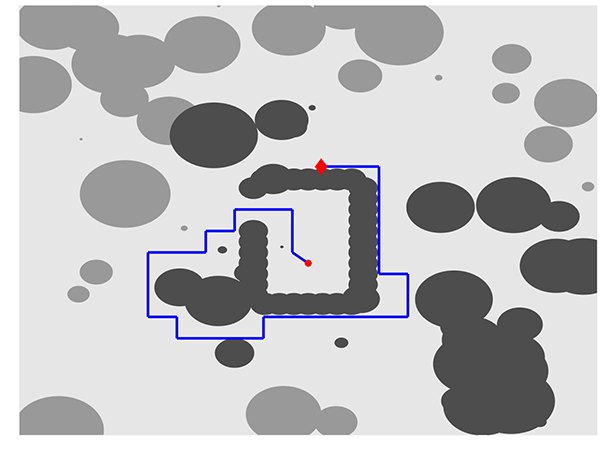} 
}
\subfigure[] { \label{fig:p36}
\includegraphics[width=0.31\columnwidth]{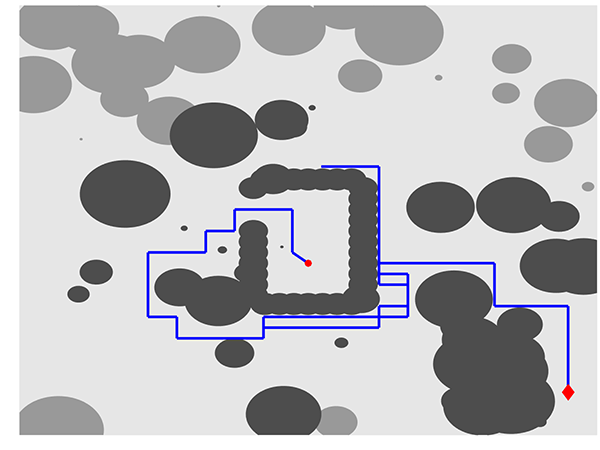} 
}
\caption{\small The paths calculated based on the results in Figure \ref{fig:g3}. (a)-(d) are middle steps that the robot stops because of the newly detected obstacles while moving and (e) is the path on which the robot get to the target. (f) gives the complete path of the robot.}
\label{fig:p3}
\end{figure}

\begin{figure} \centering
\subfigure[] { \label{fig:g41}
\includegraphics[width=0.31\columnwidth]{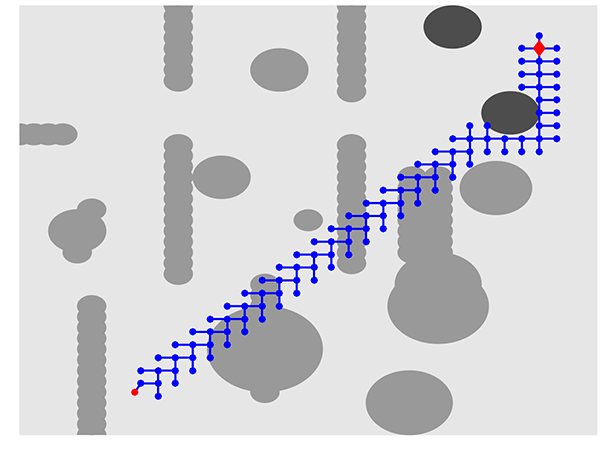} 
}
\subfigure[] { \label{fig:g42}
\includegraphics[width=0.31\columnwidth]{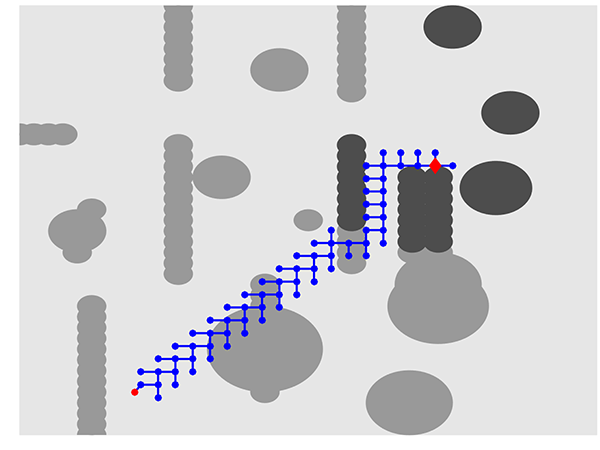} 
}
\subfigure[] { \label{fig:g43}
\includegraphics[width=0.31\columnwidth]{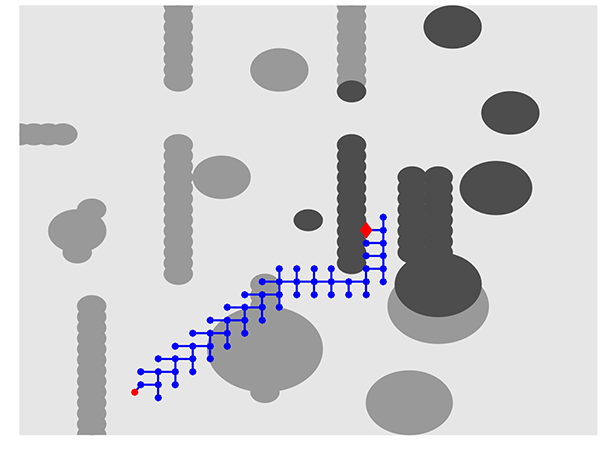} 
}
\subfigure[] { \label{fig:g44}
\includegraphics[width=0.31\columnwidth]{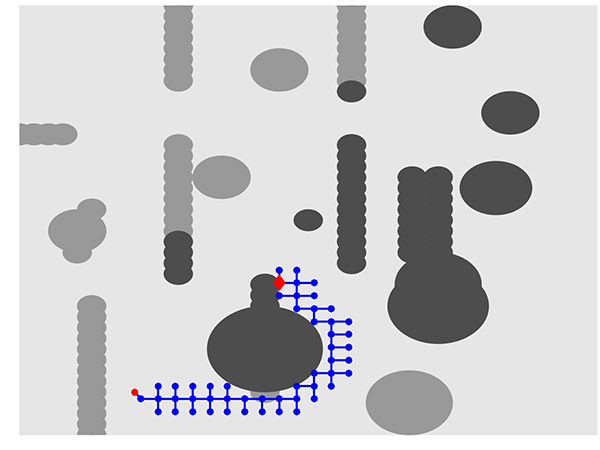} 
}
\subfigure[] { \label{fig:g45}
\includegraphics[width=0.31\columnwidth]{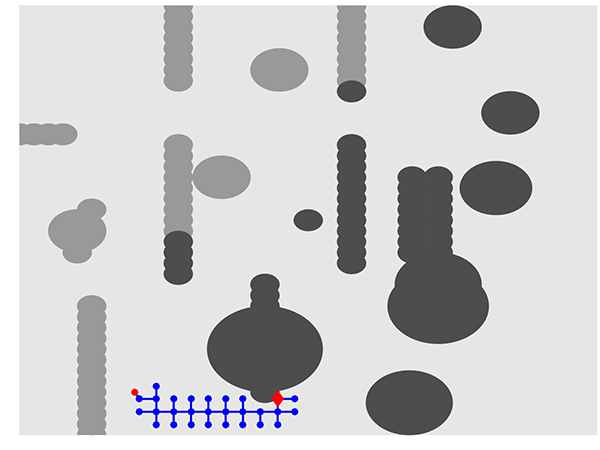} 
}
\caption{\small The generated graphs with $l=0.03$ with one robot moving in a different environment. In that environment, a narrow corridor exists and the graph is able to get through from it.}
\label{fig:g4}
\end{figure}

\begin{figure} \centering
\subfigure[] { \label{fig:p41}
\includegraphics[width=0.31\columnwidth]{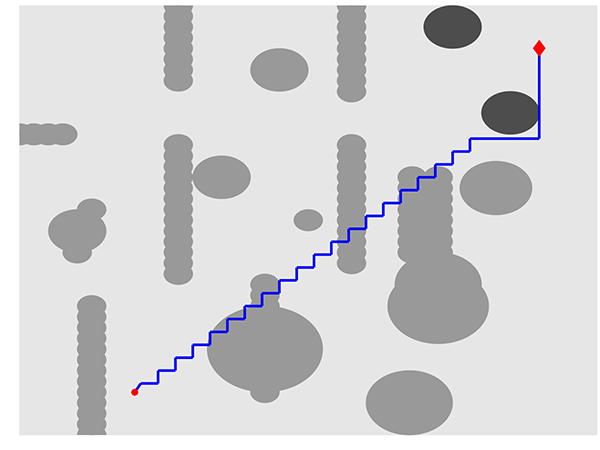} 
}
\subfigure[] { \label{fig:p42}
\includegraphics[width=0.31\columnwidth]{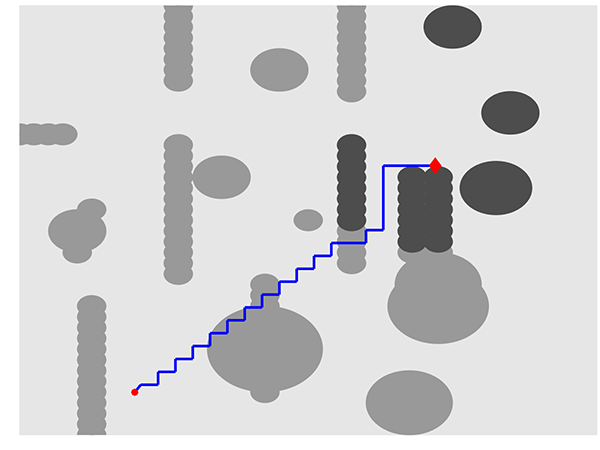} 
}
\subfigure[] { \label{fig:p43}
\includegraphics[width=0.31\columnwidth]{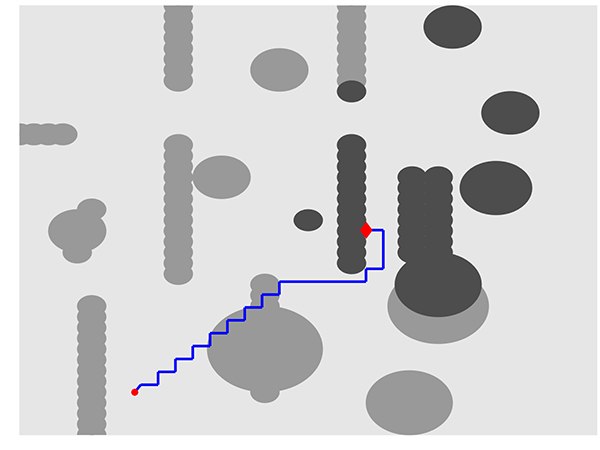} 
}
\subfigure[] { \label{fig:p44}
\includegraphics[width=0.31\columnwidth]{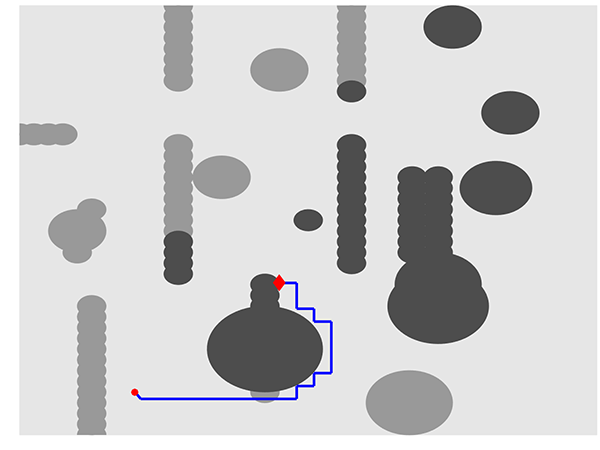} 
}
\subfigure[] { \label{fig:p45}
\includegraphics[width=0.31\columnwidth]{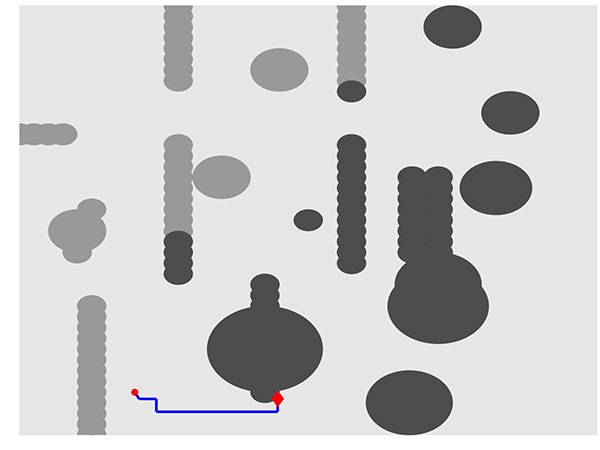} 
}
\subfigure[] { \label{fig:p46}
\includegraphics[width=0.31\columnwidth]{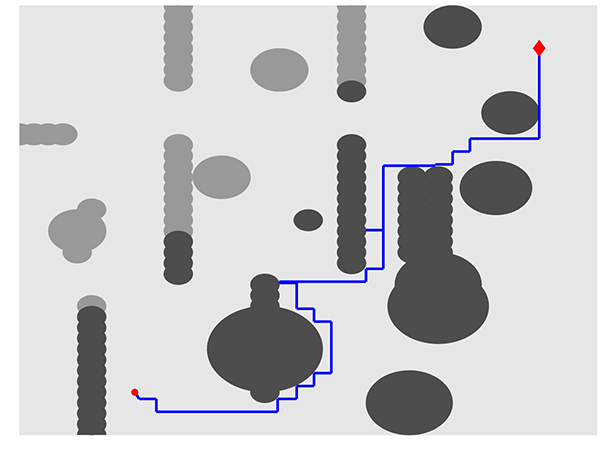} 
}
\caption{\small The paths calculated based on the results in Figure \ref{fig:g4}. (a)-(d) are middle steps that the robot stops because of the newly detected obstacles while moving and (e) is the path on which the robot get to the target. (f) gives the complete path of the robot.}
\label{fig:p4}
\end{figure}

\subsection{High dimensional cases}
In the next few examples, we calculate the paths for several multiple-agent systems. In addition to the constraints imposed by the obstacles, we also enforce that the robots cannot be too close or too far away from each other. In our examples, we set that any two robots must keep their distance between $0.03$ and $0.13$ when moving in the unknown environment. Besides, the link between each pair of robots cannot be blocked by obstacles. All examples are accompanied by youtube videos, with the web links given in the footnotes. In Figure \ref{fig:double}, a 2-robot ($4$ dimensional) system is used. From the pictures, we observe that the robots move up until trapped, because they always choose the fastest potential-decaying direction in the known environment. Then they retreat back and eventually find the correct way\footnote{video at https://youtu.be/6wKe7wnlG58}. The next example is a 3-robot system ($6$ dimensional problem)  shown in Figure \ref{fig:triple}. The environment allows a direct path from the initial to the target. The algorithm immediately finds this direct path and avoids taking other sideways\footnote{video at https://youtu.be/q84VhKfYUyo}. Finally, a 5-robot system is shown in Figure \ref{fig:five} to demonstrate that the algorithm is capable of solving a $10$-dimensional problem with complicated environment\footnote{video at https://youtu.be/H5lfzAYbfRA}, in which the robots need to twist so that they can successfully pass through the gaps between obstacles. Another example of a 10-robot system (with $20$ dimensional configuration space) moving in an unknown environment can also be found online\footnote{video at https://youtu.be/gVinTsto7pE}. We would like to note that it takes about 1 minute to finish the entire computation for this 10-robot system by using Matlab on a regular laptop (a Macbook Pro with  2.9 GHz Intel Core i5 CPU) with no particular effort being made to optimize the implementation of the algorithm.

In addition to the displayed paths, we illustrate the performance of the algorithms by using several other measurements. Table \ref{tab:numofvertices} shows the collective information about the number of vertices in graphs generated during the procedure. In this table, ``Figure'' column indicates the corresponding figures of the examples, ``num of robots'' represents the number of robots, and ``$l$'' is the generating radius in each experiment. To show the efficiency of our algorithms, we use the average number of nodes in the graphs, represented in ``avg'', and the maximum number of vertices amongst all graphs which is listed in the column ``max''. We can see that as the dimension of the problem (indicated in the ``dim'' column) increases, the size of the graphs increases, but not as fast as the exponential growth with respect to the dimensionality. 

Furthermore, we observe that the algorithm generates the particular graph with the maximum number of vertices when the robots are trapped in local minimizer (shown in ``trapped'' column). In the $6$ dimensional example, the robots do not encounter any local minimizer, which results in much fewer vertices. In fact, the sizes of graphs are smaller than those in the four-dimensional case. We also observe that the number of graphs generated by the algorithm (``num of G'' column) highly relies on the environments and the choices of the orthonormal bases. Thus it is not used as a criterion to judge the efficiency of the algorithm. Overall, our algorithm is relatively efficient especially when dealing with high dimensional problems. The most costly part is to escape the local traps, and we propose a couple of strategies to improve the performance in the next section.

\begin{table}
    \centering
    \begin{tabular}{|c|c|c|c|c|c|c|c|}\hline
        Figure & num of robots & $l$ & dim & avg & max & trapped & num of G\\\hline
        Figure \ref{fig:g1},\ref{fig:p1} & 1 & 0.03 & 2 & 60.5 & 150 & yes & 5\\\hline
        Figure \ref{fig:g2},\ref{fig:p2} & 1 & 0.03 & 2 & 81.2 & 149 & yes & 5\\\hline
        Figure \ref{fig:g3},\ref{fig:p3} & 1 & 0.05 & 2 & 47 & 94 & yes & 5\\\hline
        Figure \ref{fig:g4},\ref{fig:p4} & 1 & 0.03 & 2 & 59.2 & 104 & yes & 5\\\hline
        Figure \ref{fig:double} & 2 & 0.03 & 4 & 606.7 & 3433 & yes & 9\\\hline
        Figure \ref{fig:triple} & 3 & 0.02 & 6 & 632.4 & 1183 & no & 5\\\hline
        Figure \ref{fig:five} & 5 & 0.03 & 10 & 2178.4 & 6938 & yes & 7\\\hline
    \end{tabular}
    \caption{\small Information about number of vertices for the examples. This table shows that the number of nodes is increasing as the dimension of the problem increases but not as fast as exponential growth. And the number of nodes generated keeps small if the robots are not trapped in local minimum.}
    \label{tab:numofvertices}
\end{table}

\begin{table}
    \centering
    \begin{tabular}{|c|c|c|c|c|c|}\hline
        num of robots & $l$ & dim & avg & max & num of G\\\hline
        2 & 0.03 & 4 & 212.4 & 295 & 8\\\hline
        5 & 0.03 & 10 & 1307 & 2492 & 7\\\hline
    \end{tabular}
    \caption{\small Information about number of vertices for the examples with escaping local traps algorithm. As we can see, the algorithm generates much fewer vertices compared to Table \ref{tab:numofvertices}. And in the 5 robots system, the largest graph is no longer appears at local trap, instead the first generated graph is the with most nodes because of the physical distance from initial to target.}
    \label{tab:fewnumofvertices}
\end{table}

\begin{figure} \centering
\subfigure[] { \label{fig:double1}
\includegraphics[width=0.31\columnwidth]{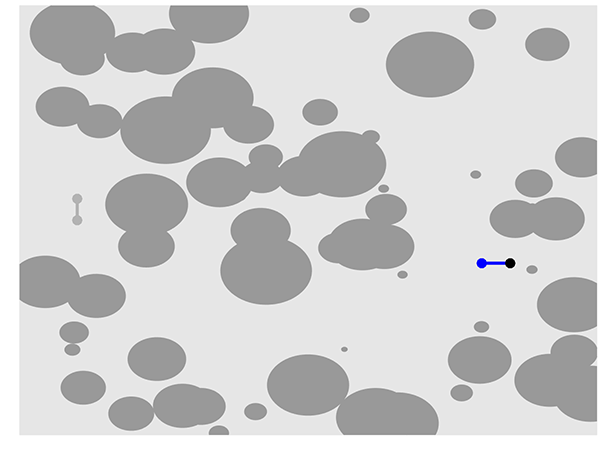} 
}
\subfigure[] { \label{fig:double2}
\includegraphics[width=0.31\columnwidth]{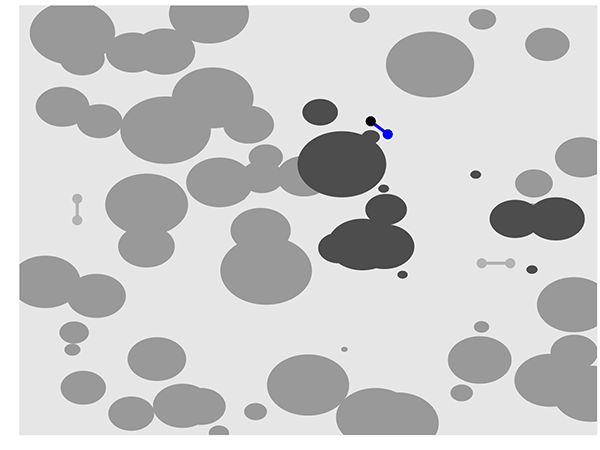} 
}
\subfigure[] { \label{fig:double3}
\includegraphics[width=0.31\columnwidth]{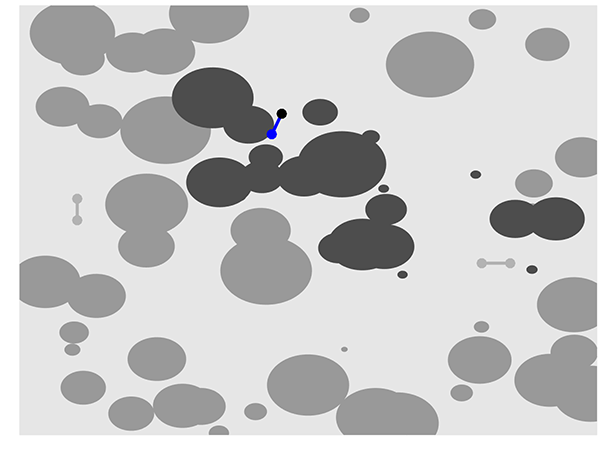} 
}
\subfigure[] { \label{fig:double4}
\includegraphics[width=0.31\columnwidth]{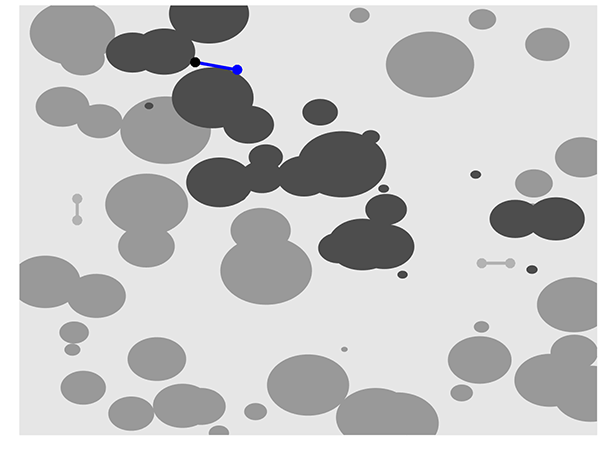} 
}
\subfigure[] { \label{fig:double5}
\includegraphics[width=0.31\columnwidth]{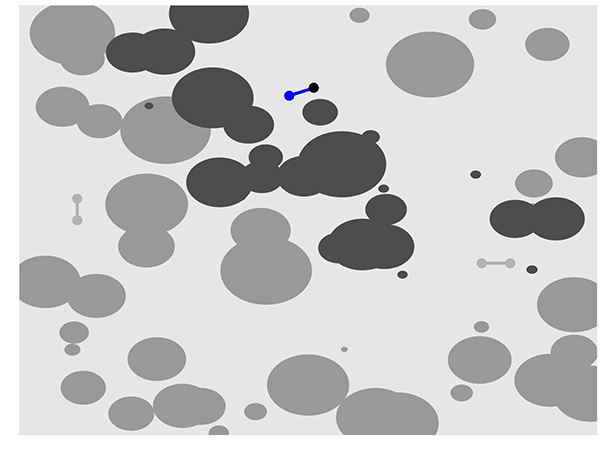} 
}
\subfigure[] { \label{fig:double6}
\includegraphics[width=0.31\columnwidth]{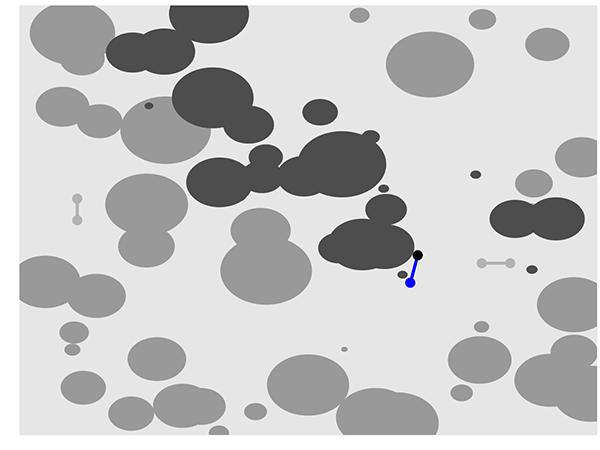} 
}
\subfigure[] { \label{fig:double7}
\includegraphics[width=0.31\columnwidth]{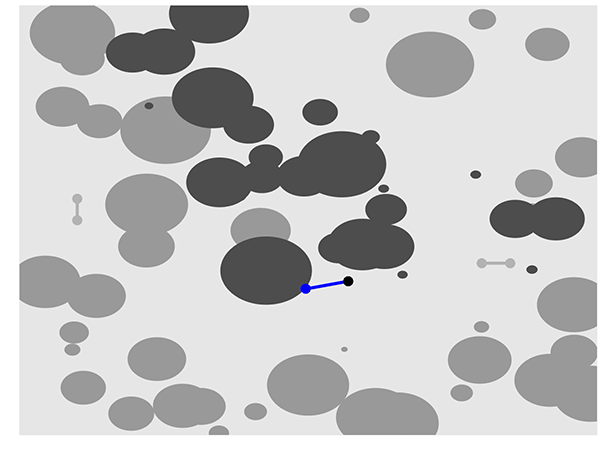} 
}
\subfigure[] { \label{fig:double8}
\includegraphics[width=0.31\columnwidth]{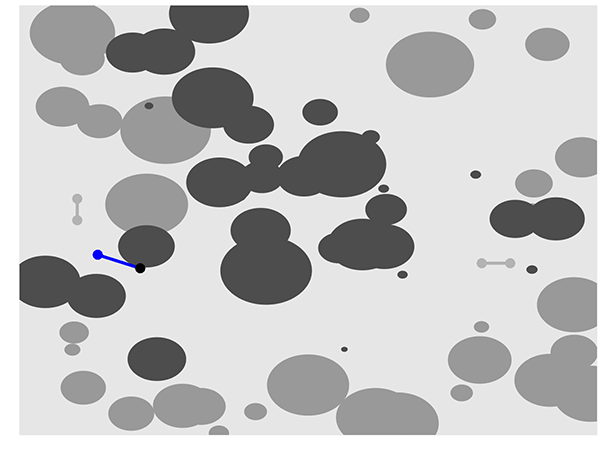} 
}
\subfigure[] { \label{fig:double9}
\includegraphics[width=0.31\columnwidth]{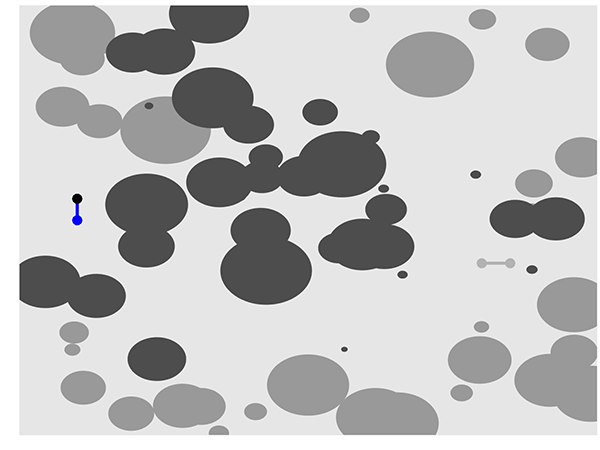} 
}
\caption{\small moving path for two robots with $l=0.03$ with linking of each two robots not blocked by obstacles. Since the environment is unknown at first, the robots choose to move from the upper side without knowing it is a dead end. After recognizing they are trapped by obstacles, the robots move down to finally find a way to the target.}
\label{fig:double}
\end{figure}

\begin{figure} \centering
\subfigure[] { \label{fig:triple1}
\includegraphics[width=0.31\columnwidth]{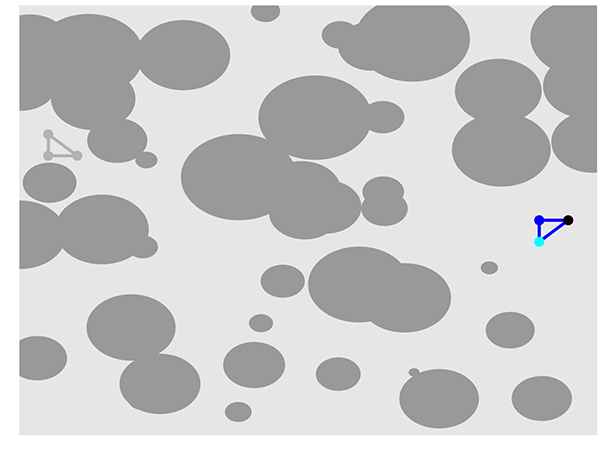} 
}
\subfigure[] { \label{fig:triple2}
\includegraphics[width=0.31\columnwidth]{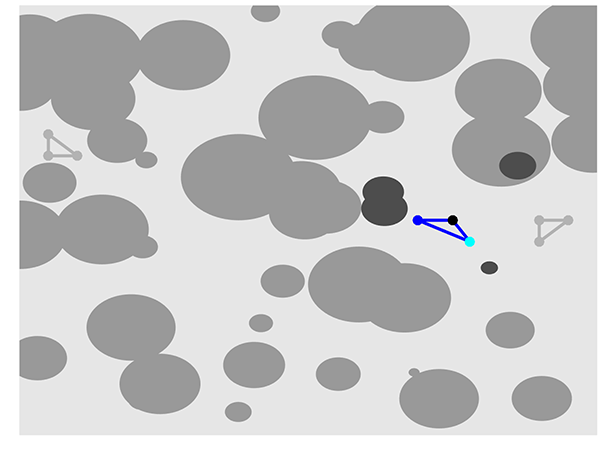} 
}
\subfigure[] { \label{fig:triple3}
\includegraphics[width=0.31\columnwidth]{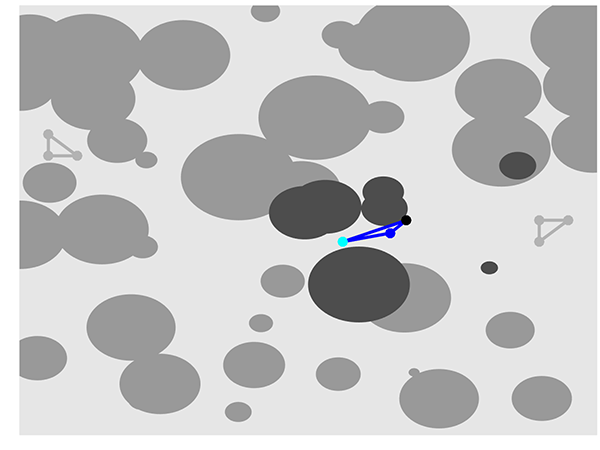} 
}
\subfigure[] { \label{fig:triple4}
\includegraphics[width=0.31\columnwidth]{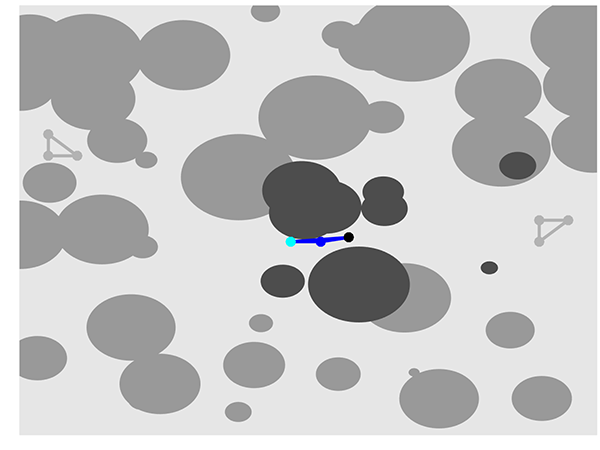} 
}
\subfigure[] { \label{fig:triple5}
\includegraphics[width=0.31\columnwidth]{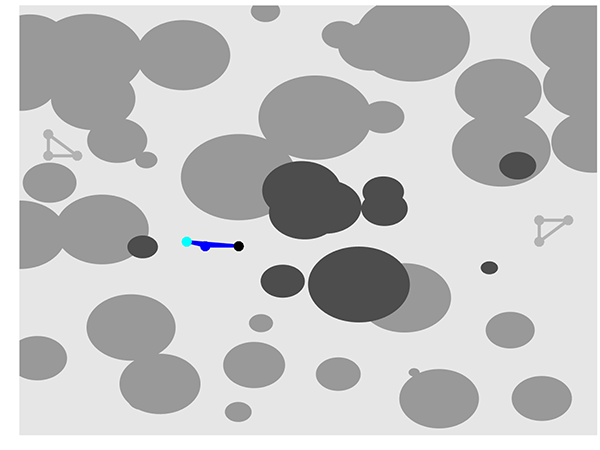} 
}
\subfigure[] { \label{fig:triple6}
\includegraphics[width=0.31\columnwidth]{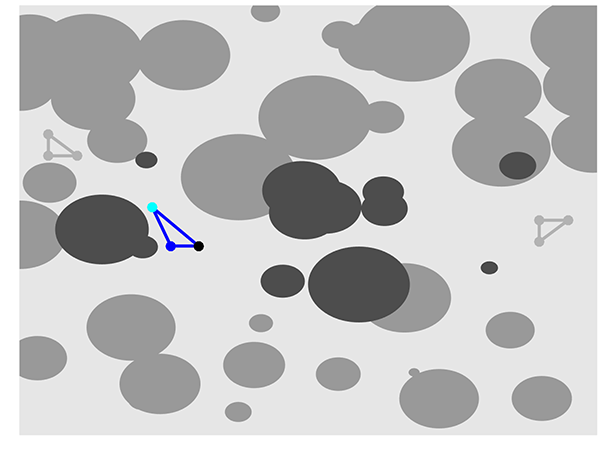} 
}
\subfigure[] { \label{fig:triple7}
\includegraphics[width=0.31\columnwidth]{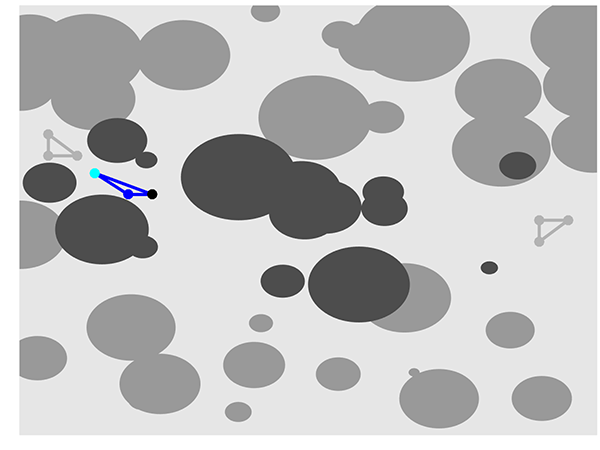} 
}
\subfigure[] { \label{fig:triple8}
\includegraphics[width=0.31\columnwidth]{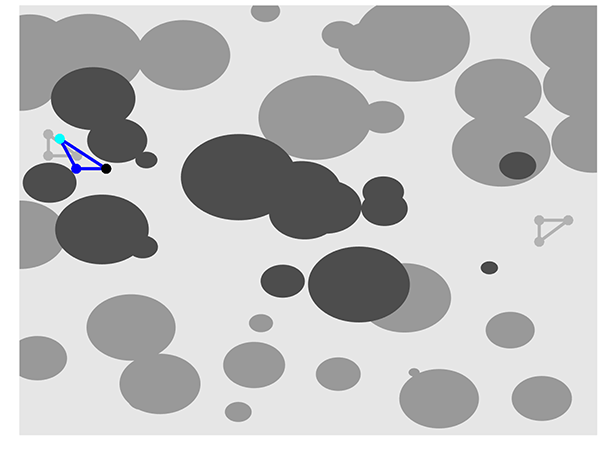} 
}
\subfigure[] { \label{fig:triple9}
\includegraphics[width=0.31\columnwidth]{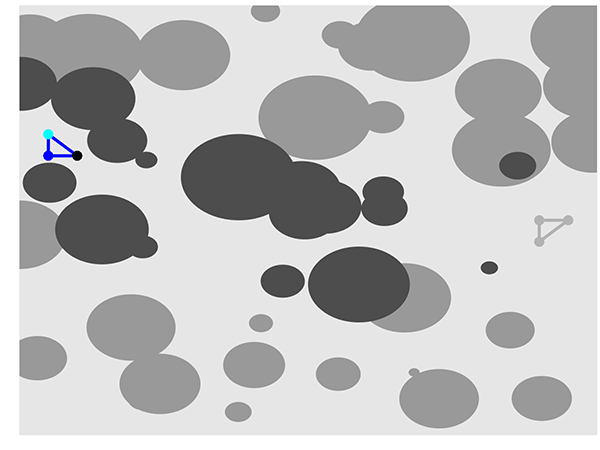} 
}
\caption{\small moving path for three robots with $l=0.02$ with linking of each two robots not blocked by obstacles. There is a direct way to get to the target and the robots successfully find it without getting into traps because the algorithm is locally greedy.}
\label{fig:triple}
\end{figure}

\begin{figure} \centering
\subfigure[] { \label{fig:five1}
\includegraphics[width=0.31\columnwidth]{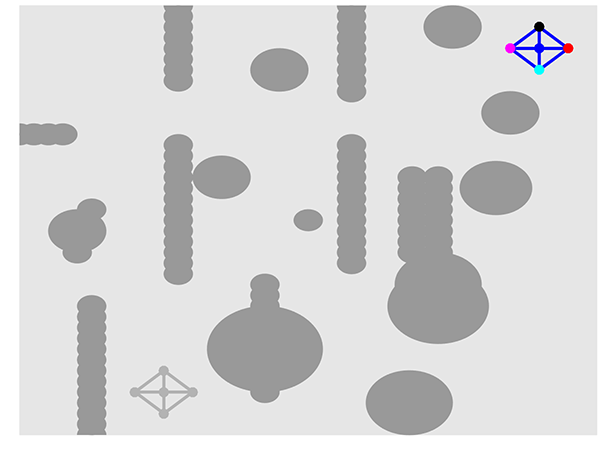} 
}
\subfigure[] { \label{fig:five2}
\includegraphics[width=0.31\columnwidth]{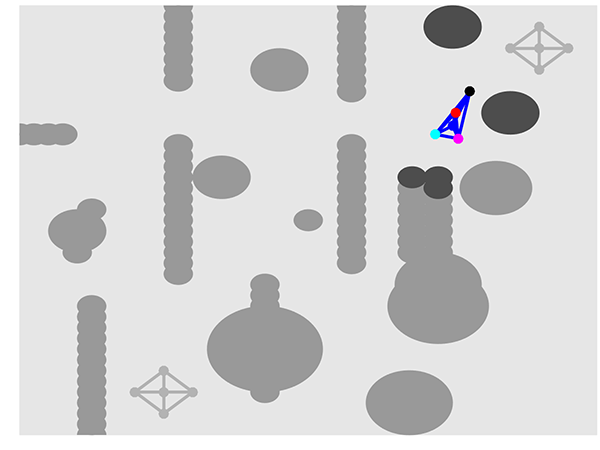} 
}
\subfigure[] { \label{fig:five3}
\includegraphics[width=0.31\columnwidth]{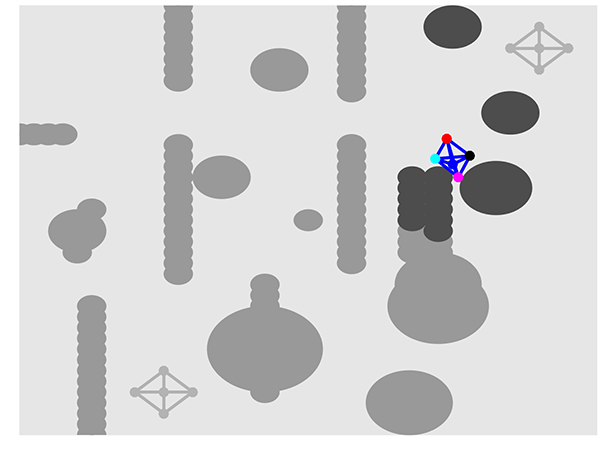} 
}
\subfigure[] { \label{fig:five4}
\includegraphics[width=0.31\columnwidth]{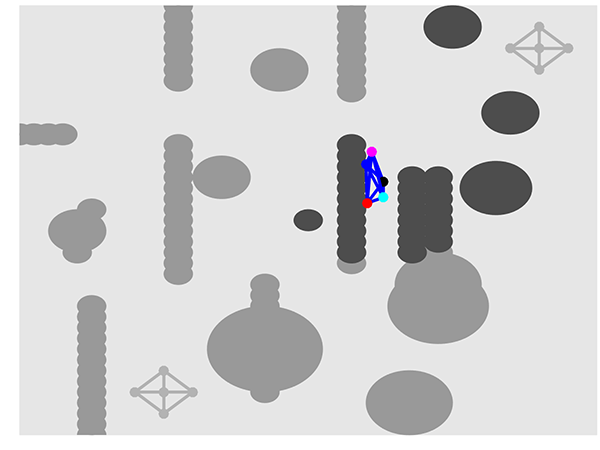} 
}
\subfigure[] { \label{fig:five5}
\includegraphics[width=0.31\columnwidth]{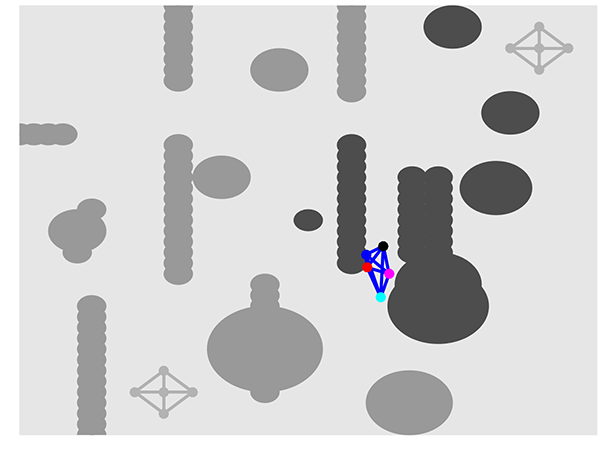} 
}
\subfigure[] { \label{fig:five6}
\includegraphics[width=0.31\columnwidth]{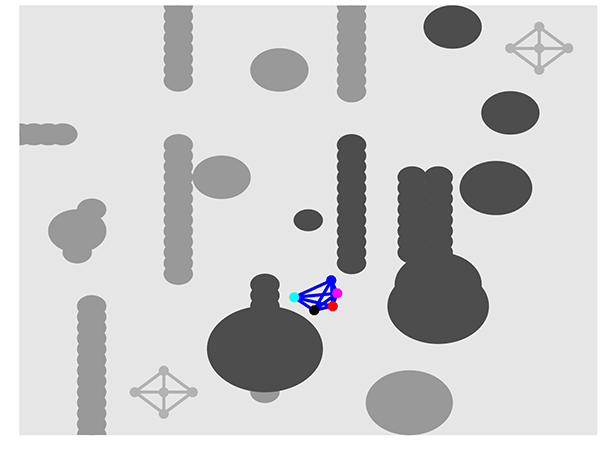} 
}
\subfigure[] { \label{fig:five7}
\includegraphics[width=0.31\columnwidth]{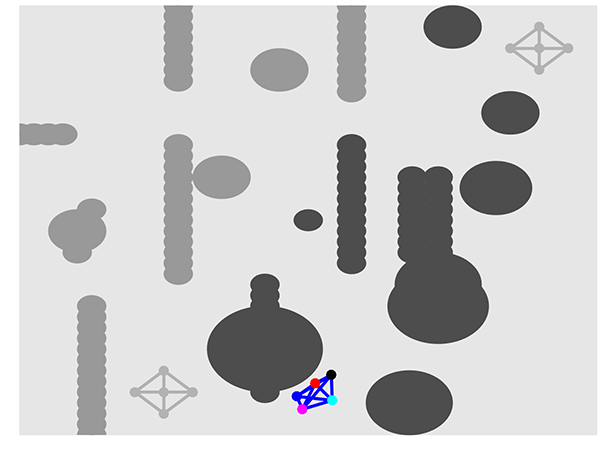} 
}
\subfigure[] { \label{fig:five8}
\includegraphics[width=0.31\columnwidth]{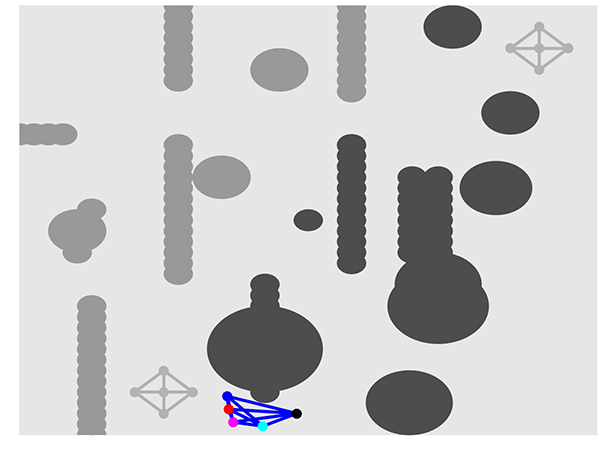} 
}
\subfigure[] { \label{fig:five9}
\includegraphics[width=0.31\columnwidth]{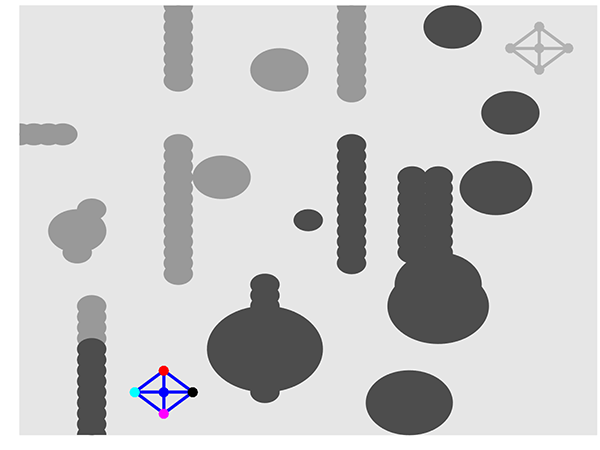} 
}
\caption{\small moving path for five robots with $l=0.03$ with linking of each two robots not blocked by obstacles. In this example, the robots change their shape to pass the narrow corridor and after getting into the local trap, they search around their way and move down to reach the destination.}
\label{fig:five}
\end{figure}

\section{Escaping Local Traps Rapidly} \label{speed_up}
From the experiments conducted, we notice that the number of generated vertices increases when the robots are trapped in local minimizers, and the number of nodes at each local trap is proportional to the volume of the trap. This is not a surprise because the nearly exhausted search is used to escape local traps. In order to reduce cost, we present two different strategies. Before doing so, we need to identify local minimizers and define their trap regions. We say that a node point $x$ is a local minimizer if no lower-potential points around $x$ can be generated by Algorithm \ref{alg:2}. Since a local trap can only be created by constraints because of the convexity of the potential function, we define the trap region as a set enclosed by the boundary of local constraints and the level curve (hyper-surface in high dimensional problems) of potential function in the following way:
\begin{align}
\mathcal{L}^*(x)=\sup_c\left\lbrace\mathcal{L}(c):\mathcal{L}(a)\text{ is closed for all }a\in [c_0,c]\right\rbrace,
\label{equ:localtrap}
\end{align}
where $c_0=p(x)$ and $\mathcal{L}(c)$ is the closed set containing $x$ with
\[
\partial\mathcal{L}(c)\subset\left(\partial\mathcal{O}_c\cup\{x:p(x)=c\}\right)
\]

When a local trap is identified, our goal is to find points located on the intersection of obstacle boundary and the level curve (surface) given by (\ref{equ:localtrap}) as quickly as possible, and then continue to generate vertices outside of the trap region. Here we introduce two different dimension reduction methods to achieve this goal.

\textbf{Keep the robot near obstacles}: We know that some of the constraints in $\phi,\psi$ must be nearly activated around the local minimizer $x\in\partial\mathcal{O}$. For the ease of presentation, we denote those nearly activated constraints as $g_i(x) \leq \epsilon$ for some integer $i$ where $\epsilon$ is a small positive number and $g_i$ is some $\phi_j$ or $\psi_k$. For example, it can be chosen as $\epsilon=\min(\min_{k=1,2,\cdots,k_1}\phi_k(x), \min_{k=1,2,\cdots,k_2}\psi_k(x))$. We modify the algorithm so that it only generates points satisfying the inequalities, that is, only add points $v$ such that $g_i(v) \leq \epsilon$ to $V$ until there is a vertex $x\in\mathcal{E}$ with
\[
\mathcal{E}=\left\lbrace x\in B(y,\sqrt{2}l):\exists z\in\{c\in(x\pm\mathcal{N})\backslash V:p(c)<p(x)\}\backslash\mathcal{O}\right\rbrace,
\]
where $\mathcal{N}$ is the substituted set of orthonormal basis for $N$ in the subspace, $V$ is the current vertex set, $y=\arg\max_{g_i(z)\leq\epsilon}p(z)$ and $g_i$ is the constraint mentioned above. After this point, we go back to Algorithm \ref{alg:2}. With the same assumptions as stated in Theorems \ref{thm1} and \ref{thm2}, we can show that this method find the path in finite steps, and the proof of the convergence follows the same arguments as provided in Section \ref{convergence}.  Since $g_i(x)=\epsilon$ is continuous locally, the modified search is conducted in a low-dimensional subspace if $\epsilon$ is chosen appropriately.

\textbf{Fix the shape formed by robots}: A different way to get out of the local traps is to introduce a set of new constraints $\{h_1(x)=0,\cdots,h_k(x)=0\}$ on the robots so that they restrict the graph generation in a low dimensional subspace $\hat{\Omega}$. For example, one may fix the pairwise distance between robots, so $h_i(x) = 0$ indicates that the distance between a certain pair of robots is a given value. In 2-D or 3-D workspace, those restrictions often lead to a fixed shape formed by the robots. Each $h_i$ reduces the search dimension by one because the new vertices added to $V$ must satisfy $h_i(x)=0$. Similar to the previous strategy, we stop this procedure when a vertex $x\in\mathcal{E}$ is generated, which indicates the robots moved out of the known local trap. On the other hand, it is possible that after adding new constraints, there is no feasible way to move out. In this case, no new vertex can be generated in $V$, then we remove one of the added constraints,  and continue with the graph generating algorithm in a subspace which is one dimension higher than the previous subspace. The procedure is repeated if necessary. For this method, if we further assume that there is a feasible tube in the low dimensional subspace defined by all constraints, including the added ones $h_i = 0$, we can use the same proof to show its convergence in a finite number of steps. In this paper, we implement this dimension reduction strategy in our high dimensional examples. 

In the two and five robots cases demonstrated in Section \ref{example}, we fix the distance between each pair of robots when a local minimizer is encountered. To compare the results, we carry out several new experiments, in which all set-ups including initial and target configurations, the obstacles and all parameters are the same. The final path and how the robots move can be found in videos\footnote{video for two robots with the improved algorithm at https://youtu.be/od5fmuo8cR8 and video for five at https://youtu.be/vVHThxmtmf8}. The information on the generated graphs is displayed in Table \ref{tab:fewnumofvertices}, where we can see that the number of vertices decreases significantly. In the 5 robots case, the largest graph is no longer produced at local traps. Instead, the first generated graph contains more nodes because of the long distance from initial to target configurations. In our examples, we observe that nodes needed around the local minimizers are reduced from $O(\alpha^n)$ to $O(\alpha^2)$, where $\alpha$ is the edge length assuming the local trap is a square and $n$ is the dimension of $\Omega$. 

We also observe a common feature in all examples: the environment is not entirely explored, and the generated graphs are greedily expanding towards the target configuration. This special feature is not by accident. In fact, it is determined by the Fokker-Planck equation in optimal transport theory. We give a thorough discussion on their connections in the next section.

\section{Relation to FPE and Optimal Transport} \label{FPE}
The design of the graph generating algorithm, Algorithm \ref{alg:2}, is inspired by the evolution of FPE, which determines a region $\mathcal{R}_f$ where the search is conducted. In this section, we describe in detail on how the region evolves following the solution of FPE,
\begin{equation}
\left\lbrace\begin{array}{l}
\frac{\partial\rho}{\partial t}(x,t)=\nabla\cdot(\rho(x,t)\nabla p(x))+\beta\Delta\rho(x,t)\\
\rho(x,0)=\rho_0(x)
\end{array}\right.,
\label{equ:fpe}
\end{equation}
where $\rho_0$ is a given distribution and $p(x)$ is the potential function. Based on \eqref{equ:fpe}, the region $\mathcal{R}_f$ is constructed by an intermittent diffusion process, meaning we take $\beta$ to be $0$, so that the density is transported greedily along the negative gradient direction, while we adjust $\beta > 0$ to trigger a diffusion process when trapped in a local minimizer. For simplicity, we call $\beta=0$ the gradient part and $\beta>0$ the diffusion part. However, since our graph generating algorithm only choose new points along the given orthonormal directions $N$, we must replace $\nabla p(x)$ in \eqref{equ:fpe}  by
its projection onto $N$:
\begin{align*}
\mathbf{u}(x)=\sum_{y\in\mathcal{H}(x)}P_y\nabla p(x),
\end{align*}
where $P_y$ is the projection operator to $y$, and $\mathcal{H}(x)$ is defined as, 
\begin{align*}
\mathcal{H}(x)=\arg\min_{y\in\mathcal{A}}\frac{\langle\nabla p(x),y\rangle}{\|\nabla p(x)\|}
\end{align*}
with $\mathcal{A}=\{y: y\in N \text{ or }-y\in N \text{ and }  x+ \lambda y \notin \mathcal{O} \forall\lambda \in (0, \xi)\text{ for some }\xi>0 \}$. The resulting equation is
\begin{equation}
\left\lbrace\begin{array}{l}
\frac{\partial\rho}{\partial t}(x,t)=\nabla\cdot(\rho(x,t)\mathbf{u}(x))+\beta\Delta\rho(x,t)\\
\rho(x,0)=\rho_0(x)
\end{array}\right.. 
\label{equ:fpegrad}
\end{equation}
We note that both \eqref{equ:fpe} and \eqref{equ:fpegrad} can be rewritten as
\begin{equation*}
\frac{\partial\rho}{\partial t}(x,t)=\nabla\cdot\left\lbrace\rho(x,t)\left[\mathbf{v}(x)+\nabla (\beta \log \rho(x,t))\right]\right\rbrace,
\end{equation*}
where $\mathbf{v}=\nabla p$ for \eqref{equ:fpe} and $\mathbf{v}=\mathbf{u}$ for \eqref{equ:fpegrad}. This expression can be approximated by the following upwind discretization of \eqref{equ:fpe} on a lattice grid $GL\subset\Omega\backslash\mathcal{O}$ (here we assume that $x_i$ is one of the grid points), with mesh size $\Delta x$ and orientation aligned with the orthonormal basis $N$ used in Algorithm \ref{alg:2} \cite{chow2012fokker},
\begin{equation}
\frac{\partial\rho_j}{\partial t}=\left(\sum_{k\in Nb(j)}(F_k(\rho,\beta)-F_j(\rho,\beta))_+\rho_kd_{jk}-\sum_{k\in Nb(j)}(F_j(\rho,\beta)-F_k(\rho,\beta))_+\rho_jd_{jk}\right)\frac{1}{(\Delta x)^2},
\label{equ:discrete}
\end{equation}
where $(\cdot)_+=\max(\cdot,0)$, $\rho_j = \rho(x_j,t)$, $Nb(j)$ is the set of all adjacent nodes (neighbors) of node $x_j$ on the grid $GL$, $F_j(\rho,\beta)=\frac{\partial}{\partial\rho_j}\mathcal{F}(\rho,\beta)$, in which $\mathcal{F}(\rho,\beta)$ is the free energy
\begin{equation} \label{equ:freeenergy}
\mathcal{F}(\rho,\beta)=\sum_{j=1}^N\left(p(x_j)\rho_j+\beta\rho_j\log\rho_j\right),
\end{equation}
and $d_{jk}=d_{kj}=1$ for \eqref{equ:fpe}. A similar discretization can be derived for \eqref{equ:fpegrad}. The value of $d_{jk}$ in the discretization of  $\mathbf{u}$, which, if assume $\phi(j)>\phi(k)$ without loss of generality, can be defined as
\begin{equation}\label{equ:proj}
d_{jk}=\left\lbrace\begin{array}{ll}1&\left\langle j-k,\nabla p(j)\right\rangle=\min_{\{i\in Nb(j):p(i)<p(j)\}}\left\langle j-i,\nabla p(j)\right\rangle\\0&\text{otherwise}\end{array}\right.,
\end{equation}
where $i,j,k$ represent the coordinates of the corresponding nodes and $\nabla p(j)$ is the gradient vector at configuration $j$ in $\Omega$. If the projection is not involved as is in the diffusion part, we simply let $d_{jk}=1$ for all $j,k$.

In the rest of this section, we show how to build $\mathcal{R}_f$ using \eqref{equ:discrete} with $\Delta x=l$. The strategy is that we alternate the procedures between the gradient ($\beta =0$) and diffusion ($\beta\neq 0$) to grow the region. When a new part of region is formed each time, we simply union it with the existing one. We want to mention that
at any point we change the procedures ($\beta$ from $0$ to a nonzero value, or vice versa), we reinitialize the density before evolving \eqref{equ:fpegrad}. We terminate the procedure, if the target configuration is included in $\mathcal{R}_f$.

\subsection{Gradient part of $\mathcal{R}_f$} \label{sub:gradpart}
For the gradient case, the points on the grid expand along the projection of the negative gradient of the potential function onto $N$. We evolve \eqref{equ:discrete} with $\beta=0$ and the initial condition 
\begin{equation*}
\rho(x,0)=\delta_{x_i}=\left\lbrace\begin{array}{ll}1&x=x_i\\0&x\neq x_i\end{array}\right.,
\end{equation*}
where $x_i$ is the starting point of the current gradient procedure. Until reaching the steady state, the solution $\rho(x,t)$ on the grid $GL$ can be calculated. The steady state solution satisfies for the following property,
\begin{proposition}\label{prop1}
$\rho(x,\infty)=\delta_{V_{loc}}$, where $V_{loc}\subset GL$ is a subset of local minimizers of $p(x)$ on the given grid.
\end{proposition}

Then we select points such that
\begin{equation}\label{equ:region_grad}
\mathcal{R}_1(x_i)=\left\lbrace x \in GL:\frac{\partial}{\partial t}\rho(x,t)>0\text{ for some }t>0\right\rbrace,
\end{equation}

Once $\mathcal{R}_1(x_i)$ is determined, we merge it to the set $\mathcal{R}$ constructed in the previous steps (we use empty at the first step), i.e. $\mathcal{R} = \mathcal{R} \bigcup \mathcal{R}_1$. If $x_i\neq x_f$, we continue to amend the set $\mathcal{R}$ with the diffusion procedure described in Section \ref{sub:diffpart}.

\subsection{Diffusion part of $\mathcal{R}_f$} \label{sub:diffpart}
We assume that the previously constructed set is $\mathcal{R}_p=\mathcal{R}$.  In the diffusion part, since $\beta >0$,  $\log\rho$ is involved in the calculation. To avoid blowing up in the computation, we initialize the density $\beta$ for \eqref{equ:discrete} as below:
\begin{equation*}
\rho(x,0)=\delta_{\mathcal{R}_p}=\left\lbrace\begin{array}{ll}\frac{1}{|\mathcal{R}_p|}(1-\epsilon)&x\in\mathcal{R}_p\\\frac{1}{|GL\backslash\mathcal{R}_p|}\epsilon&x\in GL\backslash\mathcal{R}_p\end{array}\right.,
\end{equation*}
where $\epsilon$ can be an arbitrarily small positive real number. With this initialization, we can calculate $\rho(x,t)$ in \eqref{equ:discrete} until reaching the stationary solution $\rho(x,\infty)$. Now following $\rho(x,\infty)$, we choose points on the grid as:
\begin{align*}
\mathcal{R}_2&=\bigcup_{s=0}^W\mathcal{R}_2^s,\\
\mathcal{R}_2^s&=\left\lbrace x=\arg\max_{y\in GL\backslash\bigcup_{j=0}^{s-1}\mathcal{R}_2^j}\rho(y,\infty): \mathcal{R}_2^{s-1} \bigcap Nb(x)\neq\emptyset\right\rbrace,\\
\mathcal{R}_2^0&=\mathcal{R}_p,
\end{align*}
where $W$ is defined by
\begin{equation*}
W=\arg\min_s\left\lbrace s:z\in\mathcal{R}_2^s,\exists y\in Nb(z)\backslash\left(\bigcup_{\tau=1}^s\mathcal{R}_2^{\tau}\right) \text{ with }p(y)<p(z)\right\rbrace.
\end{equation*}
We union $\mathcal{R}_2$ into $\mathcal{R}$ by defining $\mathcal{R} = \mathcal{R} \bigcup \mathcal{R}_2$. In the newly selected $\mathcal{R}_2$, we pick
\begin{equation}\label{equ:escaping}
x_i=\arg\min_y\left\lbrace p(y):y\in Nb(x)\backslash\mathcal{R}\text{ for some }x\in\mathcal{R}\right\rbrace .
\end{equation}
This $x_i$ is the new starting point for the next gradient procedure, and we return to the gradient part as described in Section \ref{sub:gradpart}.

By alternating the procedures to obtain $\mathcal{R}_1$ and $\mathcal{R}_2$ until $x_f$ is included, we define the final region
\begin{equation*}
\mathcal{R}_f=\bigcup_{x\in\mathcal{R}}Box(x,l)
\end{equation*}
where $Box(x,l)$ is the closed box centered at $x$ with edge length $2l$. With the constructed $\mathcal{R}_f$, we have the following theorem.
\begin{theorem}\label{thm3}
Assuming that the robots only stops on node points with assumptions in Theorem \ref{thm2} and $R>L$, the complete path $\gamma$ generated by the algorithm satisfies $\gamma\subset\mathcal{R}_f$.
\end{theorem}

The proofs of Theorem \ref{thm3} and Proposition \ref{prop1} are given in Section \ref{convergence}.

We show two examples in Figure \ref{fig:region} with initial configurations indicated by red diamonds and target red circles. The gray region is calculated by the gradient and diffusion procedures described in this section. The computation is done on a grid with mesh size $\Delta x=l$. As we can see clearly in Figure \ref{fig:region1}, starting from the right middle part of $\Omega$, the graph $G$ first expands along the x-axis, which is the projected negative gradient direction on the lattice grid, until it hits an obstacle. Then the procedure is switched to the diffusion case, and produces a region in front of the obstacle following the Gibbs' distribution until finding a way out.  After that, the procedure changes back to the gradient case and moves to the target greedily. This time, the projected negative gradient direction coincides with the actual one. Again, the diffusion case kicks in when a local minimizer is encountered. The procedure repeats until the target is reached.  Figure \ref{fig:region2} shows another example with more complicated set ups. 

We would like to mention that we use the forward Euler method to discretize \eqref{equ:discrete} in time,
\begin{equation}\label{equ:euler}
\rho_j^+=\rho_j+\left(\sum_{k\in Nb(j)}(F_k(\rho,\beta)-F_j(\rho,\beta))_+\rho_kd_{jk}-\sum_{k\in Nb(j)}(F_j(\rho,\beta)-F_k(\rho,\beta))_+\rho_jd_{jk}\right)\frac{\Delta t}{(\Delta x)^2}.
\end{equation}
To make the scheme convergent, we need the following proposition, which can be regarded as the Courant-Friedrichs-Lewy (CFL) conditions for numerical PDE schemes.
\begin{proposition}
Given the lattice grid with grid size $\Delta x$, \eqref{equ:euler} is stable if
\begin{equation}\label{equ:cfl}
\Delta t<\min\left\lbrace\frac{1}{\max_{j\in GL}\sum_{k\in Nb(j)}(F_j-F_k)_+d_{jk}},\min_{j\in GL}\frac{1-\rho_j}{\sum_{k\in Nb(j)}(F_k-F_j)_+\rho_kd_{jk}}\right\rbrace(\Delta x)^2,
\end{equation}
\end{proposition}

\begin{proof}
To make the scheme stable, we need
\begin{align*}
&\left\lbrace\begin{array}{l}
\rho_j+\sum_{k\in Nb(j)}(F_k-F_j)_+\rho_kd_{jk}\frac{\Delta t}{(\Delta x)^2}<1\\
\rho_j-\sum_{k\in Nb(j)}(F_j-F_k)_+\rho_jd_{jk}\frac{\Delta t}{(\Delta x)^2}>0
\end{array}\right.\\
\Longrightarrow&\left\lbrace\begin{array}{l}
\Delta t<\frac{1}{\sum_{k\in Nb(j)}(F_j-F_k)_+d_{jk}}(\Delta x)^2\\
\Delta t<\frac{1-\rho_j}{\sum_{k\in Nb(j)}(F_k-F_j)_+\rho_kd_{jk}}(\Delta x)^2
\end{array}\right.
\end{align*}
for all $j\in GL$. This leads to \eqref{equ:cfl}. In addition, following the proof of Theorem 3 in \cite{li2016study}, we can obtain $\rho_i(t)\geq\epsilon$ for a fixed grid. Therefore $(F_j-F_k)_+$ and $(F_j-F_k)_+\rho_k$ are bounded from above for all edges $\{(j,k)\}$ in $GL$. This implies that the right hand side of \eqref{equ:cfl}  is bounded from below by a positive real number, so $\Delta t$ can stay strictly positive.
\end{proof}

\begin{remark}
For each given $l$, we can get a region $\mathcal{R}_f(l)$. If we let $l$ go to $0$, which means that the lattice $GL$ approaches to the continuous space, we can define $\mathcal{R}_f(\infty)=\lim\sup_{l\rightarrow 0}\mathcal{R}_f(l)$ . The graph $G$ produced by our algorithm must satisfy $G\subset\mathcal{R}_f(\infty)$. In fact, $\mathcal{R}_f(\infty)$ is the smallest bounded region in which the search is conducted. 
\end{remark}

\begin{figure}\centering
\subfigure[] { \label{fig:region1}
\includegraphics[width=0.45\columnwidth]{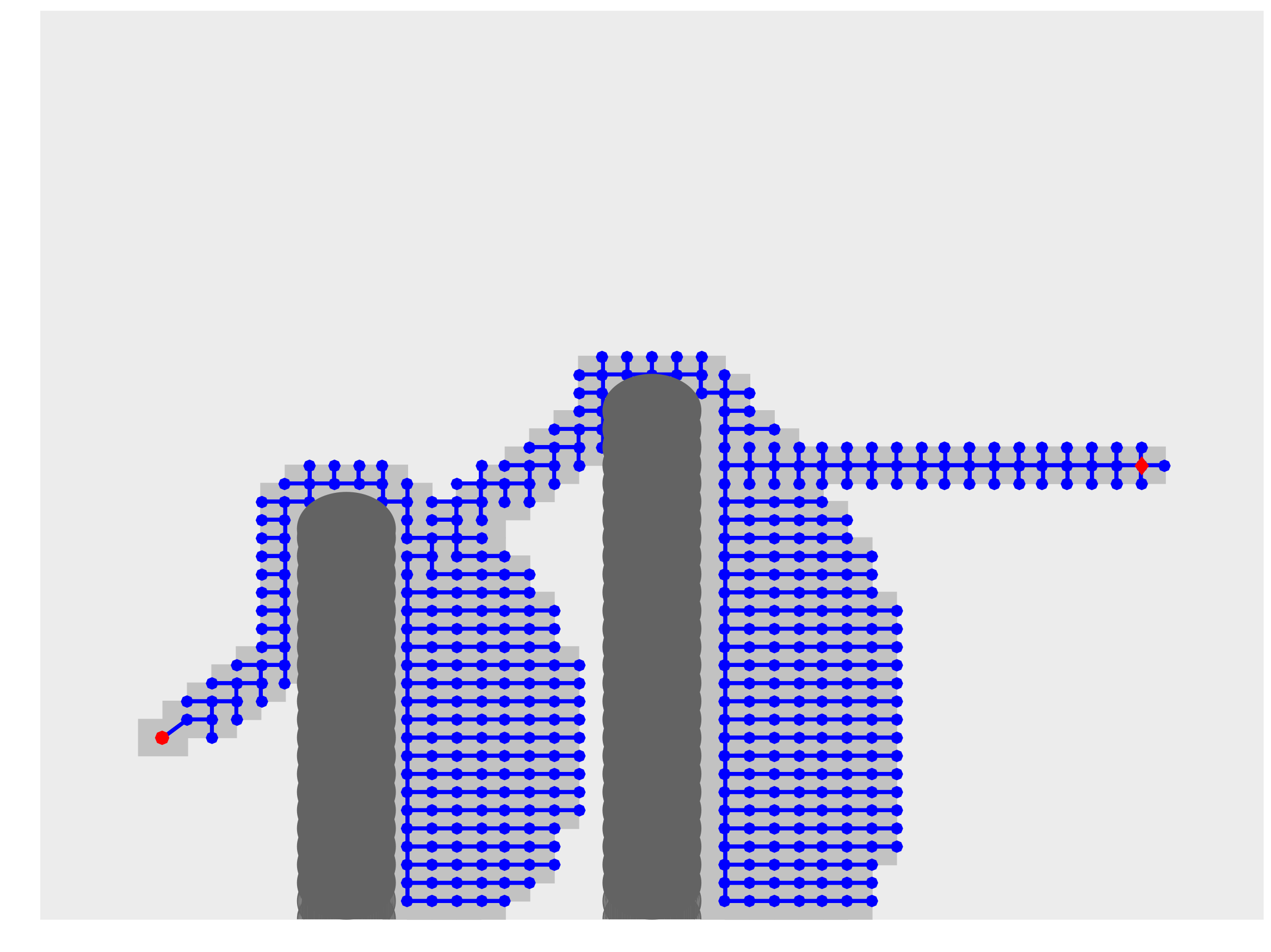} 
}
\subfigure[] { \label{fig:region2}
\includegraphics[width=0.45\columnwidth]{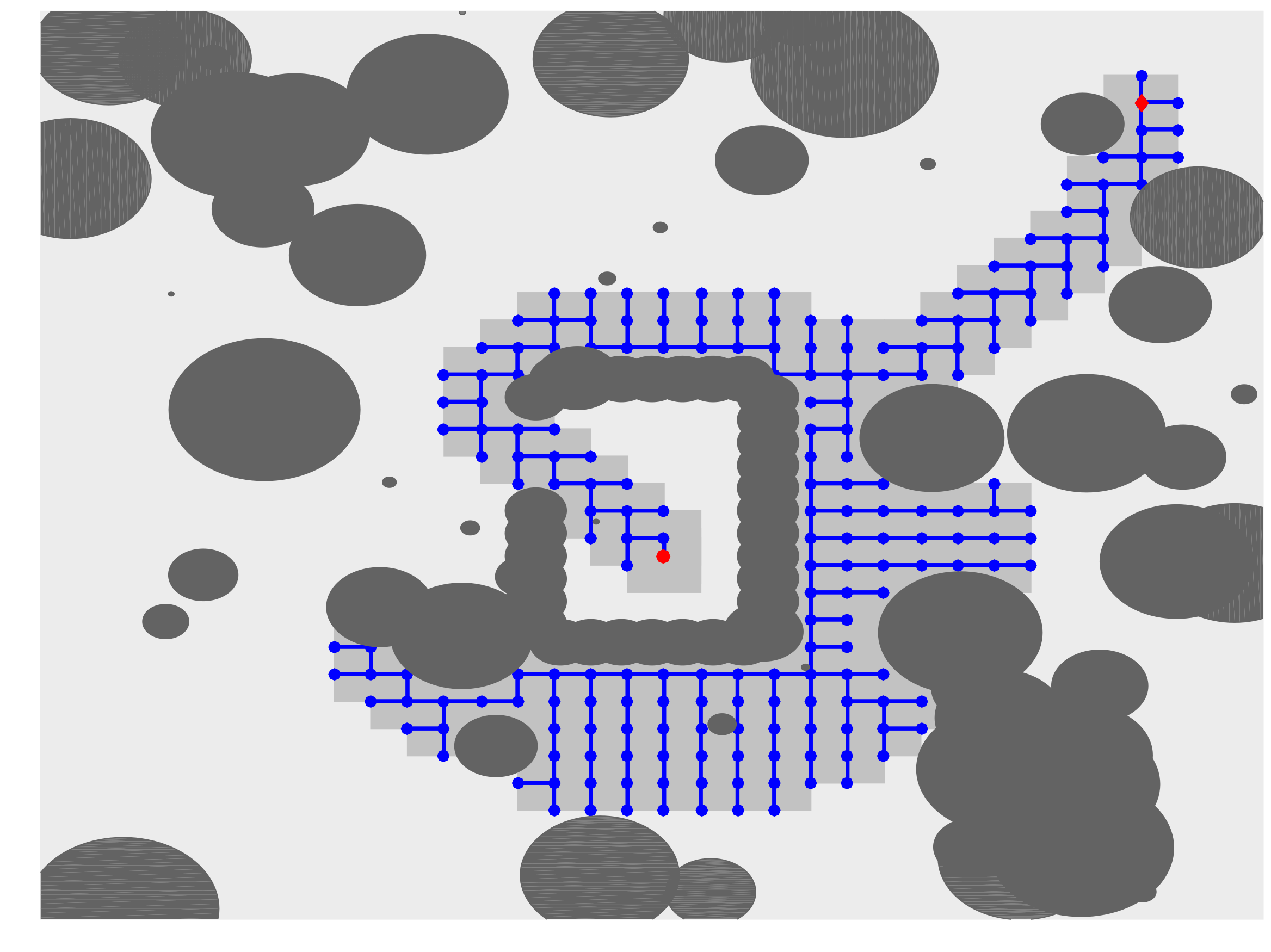} 
}
\caption{\small Relation between graph generator and FPE: Given an known environment, the graph generated by Algorithm \ref{alg:2} can be fully covered by $\mathcal{R}_f$. In both \ref{fig:region1} and \ref{fig:region2}, there exists gradient and diffusion part. The gray shadow is $\mathcal{R}_f$ and the blue part is the graph $G$ with initial and target being marked as red diamond and red dot respectively.}
\label{fig:region}
\end{figure}

\section{Convergence Analysis} \label{convergence}

\subsection{Convergence of the algorithm}
In this section, we give detailed proofs for the Theorems \ref{thm1} and \ref{thm2}. To do this, we need to prove several lemmas first. We begin with showing that the set of all feasible path $\Gamma$ is compact in a finite time interval $[0,T]$ (Lemma \ref{Lemma 1}) and there exists a feasible path in $\Gamma$ within a tubular region $\mathcal{T}$ that is clear of obstacles (Lemma \ref{Lemma 2}).

\begin{lemma} \label{Lemma 1} If there exists a feasible path,  $\Gamma$ is non-empty and compact with respect to the  $L^\infty$ norm given by 
\begin{align*}
&d_\Gamma(\gamma_1,\gamma_2)=\max_{t\in[0,T]}\|\gamma_1(t)-\gamma_2(t)\|,
\end{align*}
where $\gamma_1$ and $\gamma_2$ are two paths in $\Gamma$.
\end{lemma}

\begin{proof} Let us denote the feasible path as 
\[
\gamma:[0,T_0]\rightarrow\Omega.
\]
If we define $\gamma(t)=x_t$ for all $t\in(T_0,T]$, we have $\gamma\in\Gamma$. To prove $\Gamma$ is compact, we assume there is a sequence of paths $\{\gamma_i\}_{i=0}^\infty\subset\Gamma$ such that
\[
\lim_{i\rightarrow\infty}d_\Gamma(\gamma_i,\gamma)=0.
\]
Since $\Omega$ is compact and $\mathcal{O}$ is open, it implies that $(\Omega\backslash\mathcal{O})$ is compact. Therefore for any $t\in[0,T]$, $\gamma_i(t)\rightarrow\gamma(t)$ as $i\rightarrow\infty$, we have $\gamma(t)\in(\Omega\backslash\mathcal{O})$. This includes
\begin{align*}
&\gamma(0)=\lim_{i\rightarrow\infty}\gamma_i(0)=x_s,\\
&\gamma(T)=\lim_{i\rightarrow\infty}\gamma_i(T)=x_t,\\
&\gamma(t) \subset(\Omega\backslash\mathcal{O}),
\end{align*}
which gives $\gamma\in\Gamma$, and $\Gamma$ is a close set. In addition, since $\Omega$ is compact, $\Gamma$ is bounded, we conclude that $\Gamma$ is compact.
\end{proof}

\begin{lemma} \label{Lemma 2} Assume that (\ref{equ:feasibleregion}) is true,  there must exist a feasible path $\gamma\in\Gamma$ satisfying 
\[
\left(\cup_{t\in[0,T]}B(\gamma(t),L)\right)\cap\mathcal{O}=\emptyset,
\]
\end{lemma}
where $\mathcal{O}$ is the constrained set.

\begin{proof} Because of (\ref{equ:feasibleregion}), there exists a sequence of paths $\{\gamma_n\}\subset\Gamma$ satisfying
\[
\lim_{n\rightarrow\infty}\inf_{t\in[0,T]}\sup_{r\in[0,\infty)}\left\lbrace r:B(\gamma_n(t),r)\cap\mathcal{O}=\emptyset\right\rbrace=L.
\]
From Lemma \ref{Lemma 1}, we know that $\Gamma$ is compact, therefore $\lim_{n\rightarrow\infty}\gamma_n = \gamma_0 \in\Gamma$. For an arbitrary $t\in[0,T]$, denote
\[
L(t)=\sup_{r\in[0,\infty)}\left\lbrace r:B(\gamma_0(t),r)\cap\mathcal{O}=\emptyset\right\rbrace
\]
As the whole space $\Omega$ is compact, the limit $L(t)<\infty$. Since $\inf_{t\in[0,T]}L(t)=L$, one has $L\leq L(t)$ for arbitrary time $t\in[0,T]$. That is, fix the curve $\gamma_0$ defined as above, for all $t\in[0,T]$, $B(\gamma_0(t),L)\cap\mathcal{O}=\emptyset$ and thus
\[
\left(\cup_{t\in[0,1]}B(\gamma_0(t),L)\right)\cap\mathcal{O}=\emptyset,
\]
which proves the lemma.
\end{proof}

In the next few lemmas, we prove several results that ensure the generating graph algorithm (Algorithm \ref{alg:2}) creating new points in the feasible region when the radius $l$ is small enough compared to $L$, and the process does not stop until reaching a neighborhood of the target configuration. 

\begin{lemma} \label{Lemma 3} Given a point $x$ on an $n$-dimensional Euclidean space and $L>0$. Let $y\in S(x,L)$, and $N$ be a set containing $n$ orthonormal vectors, then $\exists z\in K=\{y\pm lN:y\pm l e, e\in N\}$ such that $z\in B(x,L)$ if
\[
0<l<\frac{2L}{\sqrt{n}}.
\]
\end{lemma}

\begin{proof} Without loss of generality, we can assume $x=(0,\cdots,0)$. Denote $y=(y_1,\cdots,y_n)$, we can rewrite $K=\{z_k=(y_1,\cdots,y_k\pm l,\cdots,y_n)\}_{k=1}^{n}$. Since $y\in S(x,L)$, then $\exists k\in\{1,\cdots,n\}$, so that $|y_k|\geq\frac{L}{\sqrt{n}}$. Consider the point $z_k=(y_1,\cdots,y_k-sign(y_k)l,\cdots,y_n)$, then
\begin{align*}
\|z_k-x\|^2&=\sum_{i\neq k}y_i^2+(y_k-sign(y_k)l)^2\\&=L^2+l^2-2sign(y_k)y_kl\\&=L^2+l^2-|y_k|l\\&\leq L^2+l^2-\frac{2L}{\sqrt{n}}l.
\end{align*}
One has $L^2+l^2-\frac{2L}{\sqrt{n}}l<L^2$ if $0<l<\frac{2L}{\sqrt{n}}$. So $z_k \in B(x,L)$. 
\end{proof}

\begin{lemma} \label{Lemma 4} Given a continuous path $\gamma$ and a closed set $\mathcal{V}\subset\Omega$ with $\gamma(0)=x,\gamma(T)=y$, and $x\in\mathcal{V}$ and $y\notin\mathcal{V}$, then there exists $z\in\gamma$ such that $z\in\partial\mathcal{V}$ and $\gamma\cap\partial\mathcal{V}$ is closed.
\end{lemma}

\begin{proof} We use the signed distance function $f(u,\mathcal{V})$ with $f>0$ if $u\in\mathcal{V}$ and $f<0$ when $u\notin\mathcal{V}$. It is clear that $f$ is continuous with respect to $u$. Hence,
\[
g=f\circ\gamma:[0,T]\rightarrow\Omega
\]
is also continuous with $g(0)\geq0$ and $g(T)\leq0$ since $\gamma(T)=y\notin\mathcal{V}$ and $\gamma(0)=x\in\mathcal{V}$. As a result, there is at least one point $t_1\in[0,T]$ so that $g(t_1)=0$, hence $\gamma(t_1)\in\partial\mathcal{V}$. In fact, all points satisfying $f(u)=0$ are on $\partial\mathcal{V}$.

Assume $\gamma\cap\partial\mathcal{V}$ is open, $\mathcal{B}=(f\circ\gamma)^{-1}(\gamma\cap\partial\mathcal{V})$ must be open, because $(f\circ\gamma)$ is continuous. This implies that $\mathcal{B}$ is the disjoint union of some open intervals. Take one of the open interval, say $(a,b)$, we have $a \notin \mathcal{B}$ and $f\circ\gamma((a,b))=0$. Due to the continuity of $f\circ\gamma$, we have $f\circ\gamma(a)=0$, which means $a \in \mathcal{B}$, and this is a contradiction. Therefore, $\gamma\cap\partial\mathcal{V}$ must be closed.
\end{proof}

\begin{lemma} \label{Lemma 5} Assume that (\ref{equ:feasibleregion}) holds. Graph $G=(V,E,K,p)$ is generated by Algorithm \ref{alg:2} with $l\leq\frac{2L}{\sqrt{n}}$. If $x_t\notin V$,
then the graph generating step of Algorithm \ref{alg:2} does not stop, and there exists at least one point in the feasible region that can be added to $G$ by the algorithm.
\end{lemma}

\begin{proof} Let us assume that the graph generating algorithm terminates after finite steps, returning a connected graph $G=(V,E)$ containing $x_s$. We denote
\[
C=\cup_{v\in V}\bar{B}(v,L),
\]
which is a closed set with $x_s\in C$. First we want to prove $x_t\in C$ by contradiction. Let us assume $x_t\notin C$. Take the path $\gamma$ in Lemma \ref{Lemma 2}, it is true that $\gamma(0)=x_s\in C$ while $\gamma(T_0)=x_t\notin C$. Since $\gamma$ is continuous, there exists at least one point that $\gamma$ intersects with $\partial C$ by Lemma \ref{Lemma 4}, and we denote it as $\{\gamma(t_i)\}$. Let $x=\sup_{t_i}\gamma(t_i)$ be the last intersection point along the path. By Algorithm \ref{alg:2}, we can find a vertex $v_c\in V$ such that $x\in S(v_c,L)$, which is the sphere centered at $v_c$ with radius $L$. Since $x\in\gamma$ and by (\ref{equ:feasibleregion}), we know that $B(x,L)\subset(\Omega\backslash\mathcal{O})$. Further we claim that there is no $v\in V\cap B(x,L)$, otherwise, $x\in B(v,L)$ implies $x\in C^o$, which is a contradiction with $x\in\partial C$.

Since the algorithm stopped, all current vertices $v\in V$ must have been tried to generate points around them. When the vertex $v_c$ is chosen, by Lemma \ref{Lemma 3}, at least one point $p$ can be generated in $B(x,L)$, which means either the algorithm should not terminate without having $p$, or $p$ already exists before, which contradicts there is no vertex of $G$ in $B(x,L)$. Therefore, we conclude that $x_t\in C$ if the algorithm stops.

If $x_t\in C$, since $x_t=\gamma(T_0)$, $B(x_t,L)\subset(\Omega\backslash\mathcal{O})$, then there is at least one vertex $v\in B(x_t,L)$, by the algorithm, an edge between $x_t$ and $v$ should be added to the graph $G$. Thus, if the algorithm stops, $x_t$ and $x_s$ are in the same connected component in the graph $G$.
\end{proof}

Now, we are ready to prove the main theorems.

\begin{proof}(Proof of Theorem \ref{thm1}) Since $\Omega$ is compact and the graph $G$ is a subset of grid points with length $l$ in $\Omega$, $G$ must contain a finite number of vertices. Otherwise, there is a cluster point in the vertex set, which implies that there exist two vertices in $G$ whose distance is strictly smaller than the generating radius $l$ regardless how small $l$ is. This contradicts to the fact that that $G$ is a subset of a grid with the smallest distance between any two points is $l$. From the construction mechanism, $G$ is always connected. Then Lemma \ref{Lemma 5} implies that $G$ connects $x_s$ and $x_t$.
\end{proof}

\begin{proof}(Proof of Theorem \ref{thm2}) We denote the detected constraints after step $i$ as $\mathcal{O}_c^i$. If the algorithm does not stop in finite steps, it means $m=\infty$. Because $\Omega$ is compact and $\mathcal{O}_c^i$ is open for every $i$, then $\Omega\backslash\mathcal{O}_c^i$ is also compact. Thus, there exists a cluster point for the stopping point set $\{\gamma_k(T_k)\}_{k=1}^\infty$. For an arbitrary $\epsilon>0$, $i$ can be chosen so that $\|\gamma_i(T_i)-\gamma_j(T_j)\|<\epsilon$ for all $j>i$. Choose $\epsilon<R-q$. Without loss of generality, one can assume that $\gamma_i$ is obtained before $\gamma_j$, and at $j$-th step, the algorithm stops because $\mathcal{O}_c\cap\gamma_j\neq\emptyset$. Even more, $$\bar{B}(\gamma_i(T_i),R)\cap\mathcal{O}_c\cap\gamma_j\neq\emptyset.$$ And since $$\bar{B}(\gamma_i(T_i),R)\cap\mathcal{O}_c=\mathcal{O}_c^i,$$ it is true that $$\mathcal{O}_c^i\cap\gamma_j\neq\emptyset$$ which is a contradiction since for each generated graph $G=(V,E,p)$ at step $j$, $$V\cap\mathcal{O}_c^{i}=\emptyset\text{ and }E\cap\mathcal{O}_c^{i}=\emptyset$$ for all $i<j$, which concludes $m<\infty$.
\end{proof}

\subsection{Proof of the bounded searching region}
In this section, we prove Theorem \ref{thm3} and related propositions. First, we give a detailed proof for the following lemma.
\begin{lemma}\label{lemma 00}
Given $\beta\geq0$, \eqref{equ:discrete} is convergent with $\mathcal{F}(\rho,\beta)$, given in \eqref{equ:freeenergy}, being a Lyapunov function. When $\beta\neq0$, \eqref{equ:discrete} converges, with any given initial condition, to the Gibbs' distribution
\begin{equation*}
\rho(x)=\frac{1}{K}\exp\left(-\frac{p(x)}{\beta}\right),
\end{equation*}
where $K$ is the normalization constant making $\rho$ a density function.
\end{lemma}
\begin{proof}
Taking the derivative along the solution $\rho(x,t)$ of \eqref{equ:discrete}, we have
\begin{align*}
\frac{d}{dt}\mathcal{F}(\rho(t))&=\sum_{j=1}^{|V|}F_j\frac{d\rho_j}{dt}\\
&=\sum_{j=1}^{|V|}\sum_{k\in Nb(j)}(F_k(\rho,\beta)-F_j(\rho,\beta))_+\rho_kd_{jk}F_j-\sum_{j=1}^{|V|}\sum_{k\in Nb(j)}(F_j(\rho,\beta)-F_k(\rho,\beta))_+\rho_jd_{jk}F_j\\
&=\sum_{j=1}^{|V|}\sum_{k\in Nb(j)}(F_j(\rho,\beta)-F_k(\rho,\beta))_+\rho_kd_{jk}F_k-\sum_{j=1}^{|V|}\sum_{k\in Nb(j)}(F_j(\rho,\beta)-F_k(\rho,\beta))_+\rho_jd_{jk}F_j\\
&=-\sum_{j=1}^{|V|}\sum_{k\in Nb(j)}(F_j(\rho,\beta)-F_k(\rho,\beta))^2_+\rho_kd_{jk}\leq0
\end{align*}
Because $\mathcal{F}(\rho,\beta)$ is bounded from below, \eqref{equ:discrete} is convergent. One can check directly that the Gibbs' distribution is the stationary solution of \eqref{equ:discrete}, and it is also the minimizer of $\mathcal{F}(\rho,\beta)$. 
\end{proof}

\begin{proof}(Proof of Proposition \ref{prop1}) First of all, by Lemma \ref{lemma 00}, \eqref{equ:discrete} converges when $\beta=0$. We assume that the support of one stationary solution $\rho(x)$ contains a point other than local minimizers, say $x_0\in GL$. Then there exists $z\in Nb(x_0)$ with $p(z)<p(x_0)$. By the definition of $d_{ij}$, there must be an edge linking $x_0$ and $z$ such that $d_{x_0z}=1$ and $p(z)<p(x_0)$. We construct
\begin{equation*}
\hat{\rho}(x)=\left\lbrace\begin{array}{ll}\rho(x)&x\neq x_0,x\neq z\\0&x=x_0\\\rho(z)+\rho(x_0)&x=z\end{array}\right.
\end{equation*}
It is easy to check $\mathcal{F}(\hat{\rho})<\mathcal{F}(\rho)$, which leads to a contradiction that $\rho$ is the minimizer of $\mathcal{F}(\rho)$.
\end{proof}

Given an arbitrary environment $\mathcal{O}$, let us denote $G_f=(V_f, E_f)$ the graph generated by Algorithm \ref{alg:2} with $\mathcal{O}_c = \mathcal{O}$ all the time.  We show that any node in $V_f$ must be in the region $\mathcal{R}_f$ constructed in Section \ref{FPE}.
\begin{lemma}\label{lemma 01}
If $x\in V_f$, then $x\in\mathcal{R}_f$. 
\end{lemma}
\begin{proof}
Since the grid size of $GL$ is $l$, we have $G_f\subset GL$,  meaning all vertices in $G_f$ are grid points on $GL$. In fact every node generated by the algorithm must be a grid point on $GL$. By the definition of $\mathcal{R}_f$, we only need to prove that $V_o\subset\mathcal{R}$, where $V_o$ is the set of vertices which are chosen to generate new nodes by Algorithm \ref{alg:2}. For convenience, we call a node in $V_o$ an interior point of $G_f$. We call $x$ the ancestor of $y$  if $y\in Nb_V(x)$, the neighborhood of $x$,  and $x$ is generated earlier than $y$ by Algorithm \ref{alg:2}. Correspondingly we call $y$ the child of $x$.  We call Algorithm \ref{alg:2} in gradient steps if the potential of the newly chosen node is lower than its ancestor, otherwise we call it in diffusion steps. We use induction to prove the lemma. Because $G_f$ contains a finite number of nodes, the induction process stops after finite steps.  At the first step, $x_s\in V_o$ and $x_s\in\mathcal{R}$. We assume that the first $k$ generated nodes in $G_f$ are contained in $\mathcal{R}$.  We want to prove that the node $z \in \mathcal{R}$, where $z$ is the node to generate new points at $(k+1)$-th step. We assume $w$ is the ancestor of $z$ that generates $w$. Obviously, $w \in \mathcal{R}$ by our induction assumption.   

\begin{enumerate}
\item  Gradient case ($p(z) < p(w)$): by the construction of $R$ in Section \ref{sub:gradpart}, to prove $z \in \mathcal{R}$, we only need to show $\rho_t(z,T) > 0$ for some $T$. There are two scenarios. (1) There exists $T_1$ such that $\rho(z,T_1) >0$. In this case, we must be able to find $T_2 \leq T_1$  with $\rho_t(z,T_2) >0$, because $\rho(z,0) = 0$ initially. (2) $\rho(z,t) = 0$ for all $t \in [0,T_1]$. Because the algorithm picks the point that has the lowest $p$ value, we have $p(w)-p(z)>p(w)-p(u)$ for all $u\in Nb(w)$, which gives $d_{wz}=1$ in \eqref{equ:euler}. Notice that $w\in\mathcal{R}$, this means there exists $T_3$ such that $\rho_t(w,T_3)>0$, which implies that there is a time interval $I = [T_3-\Delta t,T_3+\Delta t]$, $\rho_t(w,t)>0$ for all $t \in I$. Thus, there must exist a $T_4$ such that $\rho(w,T_4)>0$. By \eqref{equ:discrete}, we have $\rho_t(z,T_4)>0$ if $\rho(z,T_4)=0$, which implies $z \in \mathcal{R}$. 

\item Diffusion case ($p(z) \geq p(w)$): We recall $\mathcal{R}_2=\bigcup_{s=1}^W\mathcal{R}_2^s$. It is easy to see that if $s>t$, we have $p(x)>p(y)$ for arbitrary $x\in\mathcal{R}_2^s$ and $y\in\mathcal{R}_2^t$ because of the property of the Gibbs' distribution. Otherwise the diffusion procedure for the set construction is stopped according to the definition of $W$. We split the proof into three steps for this case.  Firstly, we let
\begin{equation*}
x=\arg\min_{z}\{p(z):z\in\bigcup_{y\in\mathcal{R}_2^W}Nb(y)\backslash\mathcal{R}_2\}.
\end{equation*}
and $x_p\in Nb(x)\cap\mathcal{R}_2$.  By the definition of $W$,  we know that $x\in\mathcal{R}$, and $x$ is the choice for the starting configuration for the next step. And also, we claim that $x$ exists. Since $x_f \in GL$ has the lowest potential and $GL$ is a connected graph, there must be a node in $\mathcal{R}$ such that one of the outside neighbor has a lower potential, otherwise the construction of $\mathcal{R}$ do not stop.

Secondly, we claim that the part of the $V_o$ generated by the diffusion steps of graph generation is contained in $\mathcal{R}_2$, and we simply use $V_o$ to denote the diffusion part generated by the algorithm. Otherwise, there exists $z\in V_o\backslash\mathcal{R}_2$ with its ancestor $w\in\mathcal{R}_2^s$ for some $s<W$. Without loss of generality, we assume $z$ is the first node outside of $\mathcal{R}_2$ chosen by Algorithm \ref{alg:2}. Because $p(z)>p(w)$, we have $z\notin\cup_{i=1}^s\mathcal{R}_2^i$. Also, $z$ is a candidate for all $\mathcal{R}_2^i$ with $i=s+1,\cdots,W$. Therefore, $z\notin\mathcal{R}_2$ means that $p(z)>p(u)$ for all $u\in\mathcal{R}_2$. By Algorithm \ref{alg:2}, this can happen only when $\mathcal{R}_2^o\subset V_o$ where
\begin{equation*}
\bar{\mathcal{R}}_2^o=\left\lbrace u:u\in\mathcal{R}_2,p(u)\leq p(x_p)\right\rbrace,
\end{equation*}
and $\mathcal{R}_2^o\subset\bar{\mathcal{R}}_2^o$ is the connected component of $\bar{\mathcal{R}}_2^o$ containing $x_p$. Otherwise there exists $v\in Nb_V(a)\backslash V_o$ for some $a\in V_o$ such that $v\in\mathcal{R}_2$. In this situation, the algorithm will choose $v$ instead of $z$ because of $p(v)<p(z)$. This cause a contradiction with the fact $z$ is the next node selected by the algorithm. On the other hand, if $\mathcal{R}_2^o\subset V_o$ holds, we have $x\in V_o$ and there exists $y\in V_o$ such that $p(y)\leq p(x)<p(z)$, leading to the fact that Algorithm \ref{alg:2} does not put $z$ into $V_o$. Thus, $V_o\subset\mathcal{R}_2\subset\mathcal{R}$.

At last, we claim that $x\in V_o$, which gives us that the start node of the next gradient step shared by both the construction of $\mathcal{R}$ and  the graph generation in the algorithm. That is, $x$ is chosen by Algorithm \ref{alg:2} to generate new points. Otherwise by the same argument for the existence of $x$, there is some node $y\in V$ such that Algorithm \ref{alg:2} generates a node $q$ based on $y$ with $p(q)<p(y)$ and $p(y)-p(q)$ maximizes the potential gap amongst all similar pairs. So, $y\in V_o$. However, by the definition of $x$, it is the lowest potential point has this property within the region $\mathcal{R}_2$. Thus $y=x$. And in the next generation step, $q$ will be generated, and meanwhile, the region growing procedure will also choose $q$ as the start point of the next gradient step.
\end{enumerate}
\end{proof}

\begin{proof}(Proof of Theorem \ref{thm3}) 
Let $G_f=(V_f,E_f)$ be the graph generated by the algorithm with $\mathcal{O}_c=\mathcal{O}$ all the time. We denote $G=(V,E)$  a graph generated with currently known environment $\mathcal{O}_c$ and initial configuration at $x$. Since $x \in G_f$ at the first step, a simple induction argument can show that every initial configuration $x$ used to generate the graph $G$ must satisfy $x \in G_f$  too. We claim that $G$ and $G_f$ must be the same in the ball $B(x,R)$, where we recall $R$ the detectable radius, i.e. $G\cap B(x,R)=G_f\cap B(x,R)$,  because $G$ and $G_f$ are generated by the same algorithm with the same knowledge of the environment in $B(x,R)$. We denote the connected common part of $G$ and $G_f$ as $G_o$, and further claim that the path $\gamma$ is in $G_o$. If not, there must exist the first node generated in the graph generation step $y\in G\backslash G_o$ on the feasible trajectory on $G$. This implies that $y\in\mathcal{O}\backslash\mathcal{O}_c$ since the ancestor of $y$ is in $G_o$, which can not contain any node in the infeasible region. However when the robots move along the path $\gamma$ to the node before $y$, the system can detect that $y$ is infeasible because $R>L$, and should stop before reaching $y$, which contradict with the fact that $y$ is on the feasible trajectory. This concludes that the trajectory of the robots is restricted on $G_o\subset G_f$ in arbitrarily given known environment $\mathcal{O}_c\subset\mathcal{O}$. By Lemma \ref{lemma 01}, $G_f$ is bounded by $\mathcal{R}_f$, hence the trajectory of the system is bounded by $\mathcal{R}_f$. 
\end{proof}

\section{Conclusion} \label{conclusion}
In this paper, an iterative algorithm is presented to solve the pathfinding problem in unknown environments. The algorithm is inspired by the solution of FPE in optimal transport theory, it contains Graph Generating, Path Finding and Environment Updating steps, among which graph generating is the key one. Guided by a potential function, the graph generating algorithm creates a tree structure, originating from the initial configuration, and aiming to the target configuration. Our approach has several advantages. First of all, the algorithm is deterministic and complete. It terminates in finite steps, returning a path, if there exists one, that is optimal with respect to the known environments at the moment of planning. Secondly, the generated graph grows linearly with respect to the dimension of the configuration space. Together with the dimension reduction strategies for escaping the local traps, our algorithm can be used to compute high dimensional problems. Lastly, as proved by using FPE and optimal transport theory, only part of the configuration space needs to be searched. We also emphasize that our assumption on the existence of a feasible path is only a technical requirement for the proof of the convergence. If such a path does not exist, the proposed algorithm can recognize the situation and stop the calculation in a finite number of iterations. In this case, one may conclude that there is not a feasible path which can be identified by using $l$ as the search step size, or reduce $l$ to further refine the path finding computation.

The proposed algorithm is our initial exploration of using optimal transport theory for path planning. Further improvements can be made on many fronts. For example, one can speed up the calculation in escaping the traps at local minimizers, by using random graph generating strategies. This is very useful in high dimensional situations.  The graph generating radius $l$ can be adaptive. Longer steps can be taken in wide open space and shorter step size is used when encountered narrow corridors. In addition, the proposed algorithms can be adapted to construct a general strategy for control problems with constraints, especially when some constraints are only available during the processes. Those are directions that are worth further studies, and we will report our findings in future papers.

\section{Acknowledgements}
	The research is partially support by grants NSF DMS-1830225, DMS-1620345, and ONR N00014-18-1-2852.
\bibliographystyle{IEEEtran}
\bibliography{ref}

\begin{thebibliography}{10}
\providecommand{\url}[1]{#1}
\csname url@samestyle\endcsname
\providecommand{\newblock}{\relax}
\providecommand{\bibinfo}[2]{#2}
\providecommand{\BIBentrySTDinterwordspacing}{\spaceskip=0pt\relax}
\providecommand{\BIBentryALTinterwordstretchfactor}{4}
\providecommand{\BIBentryALTinterwordspacing}{\spaceskip=\fontdimen2\font plus
\BIBentryALTinterwordstretchfactor\fontdimen3\font minus
  \fontdimen4\font\relax}
\providecommand{\BIBforeignlanguage}[2]{{%
\expandafter\ifx\csname l@#1\endcsname\relax
\typeout{** WARNING: IEEEtran.bst: No hyphenation pattern has been}%
\typeout{** loaded for the language `#1'. Using the pattern for}%
\typeout{** the default language instead.}%
\else
\language=\csname l@#1\endcsname
\fi
#2}}
\providecommand{\BIBdecl}{\relax}
\BIBdecl

\bibitem{overmars1992random}
M.~H. Overmars, \emph{A random approach to motion planning}.\hskip 1em plus
  0.5em minus 0.4em\relax Unknown Publisher, 1992, vol.~92.

\bibitem{kavraki1994randomized}
L.~Kavraki and J.-C. Latombe, ``Randomized preprocessing of configuration for
  fast path planning,'' in \emph{Robotics and Automation, 1994. Proceedings.,
  1994 IEEE International Conference on}.\hskip 1em plus 0.5em minus
  0.4em\relax IEEE, 1994, pp. 2138--2145.

\bibitem{amato1996randomized}
N.~M. Amato and Y.~Wu, ``A randomized roadmap method for path and manipulation
  planning,'' in \emph{Robotics and Automation, 1996. Proceedings., 1996 IEEE
  International Conference on}, vol.~1.\hskip 1em plus 0.5em minus 0.4em\relax
  IEEE, 1996, pp. 113--120.

\bibitem{svestka1997robot}
P.~Svestka, ``Robot motion planning using probabilistic roadmaps,'' \emph{PhD
  Thesis, Universiteit Utrecht}, 1997.

\bibitem{hsu1997path}
D.~Hsu, J.-C. Latombe, and R.~Motwani, ``Path planning in expansive
  configuration spaces,'' in \emph{Robotics and Automation, 1997. Proceedings.,
  1997 IEEE International Conference on}, vol.~3.\hskip 1em plus 0.5em minus
  0.4em\relax IEEE, 1997, pp. 2719--2726.

\bibitem{barraquand1997random}
J.~Barraquand, L.~Kavraki, J.-C. Latombe, R.~Motwani, T.-Y. Li, and
  P.~Raghavan, ``A random sampling scheme for path planning,'' \emph{The
  International Journal of Robotics Research}, vol.~16, no.~6, pp. 759--774,
  1997.

\bibitem{branicky2001quasi}
M.~S. Branicky, S.~M. LaValle, K.~Olson, and L.~Yang, ``Quasi-randomized path
  planning,'' in \emph{Robotics and Automation, 2001. Proceedings 2001 ICRA.
  IEEE International Conference on}, vol.~2.\hskip 1em plus 0.5em minus
  0.4em\relax IEEE, 2001, pp. 1481--1487.

\bibitem{boor1999gaussian}
V.~Boor, M.~H. Overmars, and A.~F. Van Der~Stappen, ``The gaussian sampling
  strategy for probabilistic roadmap planners,'' in \emph{Robotics and
  automation, 1999. proceedings. 1999 ieee international conference on},
  vol.~2.\hskip 1em plus 0.5em minus 0.4em\relax IEEE, 1999, pp. 1018--1023.

\bibitem{lavalle1998rapidly}
S.~M. LaValle, ``Rapidly-exploring random trees: A new tool for path
  planning,'' 1998.

\bibitem{tian2007application}
Y.~Tian, L.~Yan, G.-Y. Park, S.-H. Yang, Y.-S. Kim, S.-R. Lee, and C.-Y. Lee,
  ``Application of rrt-based local path planning algorithm in unknown
  environment,'' in \emph{Computational Intelligence in Robotics and
  Automation, 2007. CIRA 2007. International Symposium on}.\hskip 1em plus
  0.5em minus 0.4em\relax IEEE, 2007, pp. 456--460.

\bibitem{yang2013gaussian}
K.~Yang, S.~Keat~Gan, and S.~Sukkarieh, ``A gaussian process-based rrt planner
  for the exploration of an unknown and cluttered environment with a uav,''
  \emph{Advanced Robotics}, vol.~27, no.~6, pp. 431--443, 2013.

\bibitem{karaman2011sampling}
S.~Karaman and E.~Frazzoli, ``Sampling-based algorithms for optimal motion
  planning,'' \emph{The international journal of robotics research}, vol.~30,
  no.~7, pp. 846--894, 2011.

\bibitem{noreen2016optimal}
I.~Noreen, A.~Khan, and Z.~Habib, ``Optimal path planning using rrt* based
  approaches: a survey and future directions,'' \emph{Int. J. Adv. Comput. Sci.
  Appl}, vol.~7, pp. 97--107, 2016.

\bibitem{khatib1986real}
O.~Khatib, ``Real-time obstacle avoidance for manipulators and mobile robots,''
  in \emph{Autonomous robot vehicles}.\hskip 1em plus 0.5em minus 0.4em\relax
  Springer, 1986, pp. 396--404.

\bibitem{sfeir2011improved}
J.~Sfeir, M.~Saad, and H.~Saliah-Hassane, ``An improved artificial potential
  field approach to real-time mobile robot path planning in an unknown
  environment,'' in \emph{Robotic and Sensors Environments (ROSE), 2011 IEEE
  International Symposium on}.\hskip 1em plus 0.5em minus 0.4em\relax IEEE,
  2011, pp. 208--213.

\bibitem{konrad2018secant}
K.~Ahlin, ``The secant and traveling artificial potential field approaches to
  high dimensional robotic path planning,'' \emph{PhD Thesis, Georgia Institute
  of Technology}, 2018.

\bibitem{chen2016uav}
Y.-b. Chen, G.-c. Luo, Y.-s. Mei, J.-q. Yu, and X.-l. Su, ``Uav path planning
  using artificial potential field method updated by optimal control theory,''
  \emph{International Journal of Systems Science}, vol.~47, no.~6, pp.
  1407--1420, 2016.

\bibitem{montiel2015path}
O.~Montiel, U.~Orozco-Rosas, and R.~Sep{\'u}lveda, ``Path planning for mobile
  robots using bacterial potential field for avoiding static and dynamic
  obstacles,'' \emph{Expert Systems with Applications}, vol.~42, no.~12, pp.
  5177--5191, 2015.

\bibitem{lumelsky1986dynamic}
V.~Lumelsky and A.~Stepanov, ``Dynamic path planning for a mobile automaton
  with limited information on the environment,'' \emph{IEEE transactions on
  Automatic control}, vol.~31, no.~11, pp. 1058--1063, 1986.

\bibitem{lumelsky1987path}
V.~J. Lumelsky and A.~A. Stepanov, ``Path-planning strategies for a point
  mobile automaton moving amidst unknown obstacles of arbitrary shape,''
  \emph{Algorithmica}, vol.~2, no. 1-4, pp. 403--430, 1987.

\bibitem{kamon1996new}
I.~Kamon, E.~Rivlin, and E.~Rimon, ``A new range-sensor based globally
  convergent navigation algorithm for mobile robots,'' in \emph{Proceedings of
  IEEE International Conference on Robotics and Automation}, vol.~1.\hskip 1em
  plus 0.5em minus 0.4em\relax IEEE, 1996, pp. 429--435.

\bibitem{kamon1997sensory}
I.~Kamon and E.~Rivlin, ``Sensory-based motion planning with global proofs,''
  \emph{IEEE transactions on Robotics and Automation}, vol.~13, no.~6, pp.
  814--822, 1997.

\bibitem{mcguire2018comparative}
K.~McGuire, G.~de~Croon, and K.~Tuyls, ``A comparative study of bug algorithms
  for robot navigation,'' \emph{arXiv preprint arXiv:1808.05050}, 2018.

\bibitem{ng2007performance}
J.~Ng and T.~Br{\"a}unl, ``Performance comparison of bug navigation
  algorithms,'' \emph{Journal of Intelligent and Robotic Systems}, vol.~50,
  no.~1, pp. 73--84, 2007.

\bibitem{doran1966experiments}
J.~E. Doran and D.~Michie, ``Experiments with the graph traverser program,''
  \emph{Proceedings of the Royal Society of London. Series A. Mathematical and
  Physical Sciences}, vol. 294, no. 1437, pp. 235--259, 1966.

\bibitem{stentz1994optimal}
A.~Stentz, ``Optimal and efficient path planning for partially-known
  environments,'' in \emph{ICRA}, vol.~94, 1994, pp. 3310--3317.

\bibitem{stentz1995focussed}
A.~Stentz \emph{et~al.}, ``The focussed d\^{}* algorithm for real-time
  replanning,'' in \emph{IJCAI}, vol.~95, 1995, pp. 1652--1659.

\bibitem{koenig2005fast}
S.~Koenig and M.~Likhachev, ``Fast replanning for navigation in unknown
  terrain,'' \emph{IEEE Transactions on Robotics}, vol.~21, no.~3, pp.
  354--363, 2005.

\bibitem{podsedkowski2001new}
L.~Podsedkowski, J.~Nowakowski, M.~Idzikowski, and I.~Vizvary, ``A new solution
  for path planning in partially known or unknown environment for nonholonomic
  mobile robots,'' \emph{Robotics and Autonomous Systems}, vol.~34, no. 2-3,
  pp. 145--152, 2001.

\bibitem{walker2002comparison}
M.~Walker and C.~H. Messom, ``A comparison of genetic programming and genetic
  algorithms for auto-tuning mobile robot motion control,'' in \emph{Electronic
  Design, Test and Applications, 2002. Proceedings. The First IEEE
  International Workshop on}.\hskip 1em plus 0.5em minus 0.4em\relax IEEE,
  2002, pp. 507--509.

\bibitem{lei2006improved}
L.~Lei, H.~Wang, and Q.~Wu, ``Improved genetic algorithms based path planning
  of mobile robot under dynamic unknown environment,'' in \emph{Mechatronics
  and Automation, Proceedings of the 2006 IEEE International Conference
  on}.\hskip 1em plus 0.5em minus 0.4em\relax IEEE, 2006, pp. 1728--1732.

\bibitem{contreras2015mobile}
M.~A. Contreras-Cruz, V.~Ayala-Ramirez, and U.~H. Hernandez-Belmonte, ``Mobile
  robot path planning using artificial bee colony and evolutionary
  programming,'' \emph{Applied Soft Computing}, vol.~30, pp. 319--328, 2015.

\bibitem{hassanzadeh2009path}
I.~Hassanzadeh and S.~M. Sadigh, ``Path planning for a mobile robot using fuzzy
  logic controller tuned by ga,'' in \emph{Mechatronics and its Applications,
  2009. ISMA'09. 6th International Symposium on}.\hskip 1em plus 0.5em minus
  0.4em\relax IEEE, 2009, pp. 1--5.

\bibitem{wang2005fuzzy}
M.~Wang \emph{et~al.}, ``Fuzzy logic based robot path planning in unknown
  environment,'' in \emph{Machine Learning and Cybernetics, 2005. Proceedings
  of 2005 International Conference on}, vol.~2.\hskip 1em plus 0.5em minus
  0.4em\relax IEEE, 2005, pp. 813--818.

\bibitem{luo2008bioinspired}
C.~Luo and S.~X. Yang, ``A bioinspired neural network for real-time concurrent
  map building and complete coverage robot navigation in unknown
  environments,'' \emph{IEEE Transactions on Neural Networks}, vol.~19, no.~7,
  pp. 1279--1298, 2008.

\bibitem{ersson2001path}
T.~Ersson and X.~Hu, ``Path planning and navigation of mobile robots in unknown
  environments.'' in \emph{IROS}, 2001, pp. 858--864.

\bibitem{li2017method}
W.~Li, J.~Lu, H.~Zhou, and S.-N. Chow, ``Method of evolving junctions: A new
  approach to optimal control with constraints,'' \emph{Automatica}, vol.~78,
  pp. 72--78, 2017.

\bibitem{janson2015fast}
L.~Janson, E.~Schmerling, A.~Clark, and M.~Pavone, ``Fast marching tree: A fast
  marching sampling-based method for optimal motion planning in many
  dimensions,'' \emph{The International journal of robotics research}, vol.~34,
  no.~7, pp. 883--921, 2015.

\bibitem{dolgov2010path}
D.~Dolgov, S.~Thrun, M.~Montemerlo, and J.~Diebel, ``Path planning for
  autonomous vehicles in unknown semi-structured environments,'' \emph{The
  International Journal of Robotics Research}, vol.~29, no.~5, pp. 485--501,
  2010.

\bibitem{van2006anytime}
J.~Van Den~Berg, D.~Ferguson, and J.~Kuffner, ``Anytime path planning and
  replanning in dynamic environments,'' in \emph{Robotics and Automation, 2006.
  ICRA 2006. Proceedings 2006 IEEE International Conference on}.\hskip 1em plus
  0.5em minus 0.4em\relax IEEE, 2006, pp. 2366--2371.

\bibitem{wagner2015subdimensional}
G.~Wagner and H.~Choset, ``Subdimensional expansion for multirobot path
  planning,'' \emph{Artificial Intelligence}, vol. 219, pp. 1--24, 2015.

\bibitem{cui2016mutual}
R.~Cui, Y.~Li, and W.~Yan, ``Mutual information-based multi-auv path planning
  for scalar field sampling using multidimensional rrt,'' \emph{IEEE
  Transactions on Systems, Man, and Cybernetics: Systems}, vol.~46, no.~7, pp.
  993--1004, 2016.

\bibitem{das2016hybridization}
P.~Das, H.~S. Behera, and B.~K. Panigrahi, ``A hybridization of an improved
  particle swarm optimization and gravitational search algorithm for
  multi-robot path planning,'' \emph{Swarm and Evolutionary Computation},
  vol.~28, pp. 14--28, 2016.

\bibitem{villani2008optimal}
C.~Villani, \emph{Optimal transport: old and new}.\hskip 1em plus 0.5em minus
  0.4em\relax Springer Science \& Business Media, 2008, vol. 338.

\bibitem{kantarovich1939mathematical}
L.~Kantarovich, ``Mathematical methods in the organization and planning of
  production,'' \emph{Publication House of the Leningrad State
  University.[Translated in Management Sc. vol 66, 366-422]}, 1939.

\bibitem{brenier1987decomposition}
Y.~Brenier, ``D{\'e}composition polaire et r{\'e}arrangement monotone des
  champs de vecteurs,'' \emph{CR Acad. Sci. Paris S{\'e}r. I Math.}, vol. 305,
  pp. 805--808, 1987.

\bibitem{brenier1991polar}
------, ``Polar factorization and monotone rearrangement of vector-valued
  functions,'' \emph{Communications on pure and applied mathematics}, vol.~44,
  no.~4, pp. 375--417, 1991.

\bibitem{otto2001geometry}
F.~Otto, ``The geometry of dissipative evolution equations: the porous medium
  equation,'' 2001.

\bibitem{bandyopadhyay2014probabilistic}
S.~Bandyopadhyay, S.-J. Chung, and F.~Y. Hadaegh, ``Probabilistic swarm
  guidance using optimal transport,'' in \emph{2014 IEEE Conference on Control
  Applications (CCA)}.\hskip 1em plus 0.5em minus 0.4em\relax IEEE, 2014, pp.
  498--505.

\bibitem{krishnan2018distributed}
V.~Krishnan and S.~Mart{\'\i}nez, ``Distributed optimal transport for the
  deployment of swarms,'' in \emph{2018 IEEE Conference on Decision and Control
  (CDC)}.\hskip 1em plus 0.5em minus 0.4em\relax IEEE, 2018, pp. 4583--4588.

\bibitem{dijkstra1959note}
E.~W. Dijkstra, ``A note on two problems in connexion with graphs,''
  \emph{Numerische mathematik}, vol.~1, no.~1, pp. 269--271, 1959.

\bibitem{zuse1972plankalkul}
K.~Zuse, \emph{Der Plankalk{\"u}l}.\hskip 1em plus 0.5em minus 0.4em\relax
  Gesellschaft f{\"u}r Mathematik und Datenverarbeitung, 1972, no.~63.

\bibitem{moore1959shortest}
E.~F. Moore, ``The shortest path through a maze,'' in \emph{Proc. Int. Symp.
  Switching Theory, 1959}, 1959, pp. 285--292.

\bibitem{chow2012fokker}
S.-N. Chow, W.~Huang, Y.~Li, and H.~Zhou, ``Fokker--planck equations for a free
  energy functional or markov process on a graph,'' \emph{Archive for Rational
  Mechanics and Analysis}, vol. 203, no.~3, pp. 969--1008, 2012.

\bibitem{li2016study}
W.~Li, ``A study of stochastic differential equations and fokker-planck
  equations with applications,'' Ph.D. dissertation, Georgia Institute of
  Technology, 2016.

\end{thebibliography}

\end{document}